\documentclass[10pt,twoside,a4paper]{article}

\usepackage[utf8]{inputenc}

\usepackage[T1]{fontenc}

\usepackage{amsmath}
\numberwithin{equation}{section}

\usepackage{amssymb}

\usepackage{microtype}

\usepackage[backend=bibtex,maxnames=10,backref=true,hyperref=true,safeinputenc]{biblatex}
\DefineBibliographyStrings{english}{%
	backrefpage = {cited on page},
	backrefpages = {cited on pages},
}

\DeclareFieldFormat[report]{title}{``#1''}
\DeclareFieldFormat[book]{title}{``#1''}
\AtEveryBibitem{\clearfield{url}}
\AtEveryBibitem{\clearfield{note}}

\usepackage{graphicx}
\graphicspath{{images/}}

\usepackage{tikz}
\usepackage{pgfplots}
\usetikzlibrary{intersections,calc,decorations.pathreplacing}

\newcommand{\ourtitle}{Singular Values of the Attenuated Photoacoustic Imaging Operator}

\title{\ourtitle}
\author{Peter Elbau$^1$\\{\footnotesize\href{mailto:peter.elbau@univie.ac.at}{peter.elbau@univie.ac.at}}
\and Otmar Scherzer$^{1,2}$\\{\footnotesize\href{mailto:otmar.scherzer@univie.ac.at}{otmar.scherzer@univie.ac.at}}
\and Cong Shi$^1$\\{\footnotesize\href{mailto:cong.shi@univie.ac.at}{cong.shi@univie.ac.at}}}
\date{}

\usepackage[pdftex,colorlinks=true,linkcolor=blue,citecolor=green,urlcolor=blue,bookmarks=true,bookmarksnumbered=true]{hyperref}
\makeatletter
\AtBeginDocument{
	\hypersetup
	{
	    pdfauthor={Peter Elbau, Otmar Scherzer, Cong Shi},
	    pdfsubject={Attenuation in Photoacoustic Imaging},
	    pdftitle={\ourtitle},
	    pdfkeywords={Photoacoustic Imaging, Spherical Mean, Attenuation, Singular Values}
	}
}
\makeatother

\makeatletter
\let\oldtheequation\theequation
\renewcommand\tagform@[1]{\maketag@@@{\ignorespaces#1\unskip\@@italiccorr}}
\renewcommand\theequation{(\oldtheequation)}
\makeatother
\def\equationautorefname~{}

\usepackage[hyperref,amsmath,thmmarks]{ntheorem}
\usepackage{aliascnt}

\newtheorem{lemma}{Lemma}[section]

\newaliascnt{proposition}{lemma}
\newtheorem{proposition}[proposition]{Proposition}
\aliascntresetthe{proposition}

\newaliascnt{corollary}{lemma}
\newtheorem{corollary}[corollary]{Corollary}
\aliascntresetthe{corollary}

\newaliascnt{theorem}{lemma}
\newtheorem{theorem}[theorem]{Theorem}
\aliascntresetthe{theorem}

\theorembodyfont{\normalfont}
\newaliascnt{definition}{lemma}
\newtheorem{definition}[definition]{Definition}
\aliascntresetthe{definition}

\newaliascnt{assumption}{lemma}

\aliascntresetthe{assumption}

\theoremstyle{nonumberplain}
\theoremseparator{:}
\theoremheaderfont{\normalfont\itshape}

\newtheorem{remark}{Remark}

\theoremsymbol{\ensuremath{\square}}
\newtheorem{proof}{Proof}

\usepackage[a4paper,centering,bindingoffset=0cm,marginpar=2cm,margin=2.5cm]{geometry}

\usepackage[pagestyles]{titlesec}
\titleformat{\section}[block]{\large\sc\filcenter}{\thesection.}{0.5ex}{}[]
\titleformat{\subsection}[runin]{\bf}{\thesubsection.}{0.5ex}{}[.]

\makeatletter
\AtBeginDocument{
\newpagestyle{headers}%
{%
	\headrule
	\sethead%
		[\footnotesize\thepage]%
		[\footnotesize\sc Peter Elbau, Otmar Scherzer, Cong Shi]%
		[]%
		{}%
		{\footnotesize\sc\ourtitle}%
		{\footnotesize\thepage}
	\setfoot{}{}{}
}
\pagestyle{headers}}
\makeatother

\usepackage[font=footnotesize,format=plain,labelfont=sc,textfont=sl,width=0.75\textwidth,labelsep=period]{caption}

\usepackage{dsfont}
\newcommand{\N}{\mathds{N}}

\newcommand{\R}{\mathds{R}}
\newcommand{\C}{\mathds{C}}
\renewcommand{\H}{\mathds{H}}

\let\RE\Re
\let\Re=\undefined
\DeclareMathOperator{\Re}{\RE e}
\let\IM\Im
\let\Im=\undefined
\DeclareMathOperator{\Im}{\IM m}
\DeclareMathOperator{\supp}{supp}

\DeclareMathOperator{\dist}{dist}
\DeclareMathOperator{\diam}{diam}
\DeclareMathOperator{\Rank}{rank}

\newcommand{\e}{\mathrm e}
\let\ii\i
\renewcommand{\i}{\mathrm i}
\renewcommand{\d}{\,\mathrm d}
\newcommand{\Dj}{\left.\frac{\partial^j}{\partial s^j}\right|_{s=0}\!\!\!\!\!\!}
\newcommand{\Ord}{\mathcal O}
\newcommand{\PO}{\check{\mathcal P}_\kappa}
\newcommand{\POa}{\check{\mathcal P}_\kappa^{(0)}{}}
\newcommand{\POb}{\check{\mathcal P}_\kappa^{(1)}{}}
\newcommand{\POK}{F_\kappa}
\newcommand{\POKa}{F_\kappa^{(0)}}
\newcommand{\POka}{f_\kappa^{(0)}}
\newcommand{\POKb}{F_\kappa^{(1)}}
\newcommand{\POkb}{f_\kappa^{(1)}}
\newcommand{\POKc}{F_\kappa^{(2)}}
\newcommand{\POkc}{f_\kappa^{(2)}}

\begin{document}

\maketitle
\thispagestyle{empty}
\hspace*{1em}
\parbox[t]{0.49\textwidth}{\footnotesize
\hspace*{-1ex}$^1$Computational Science Center\\
University of Vienna\\
Oskar-Morgenstern-Platz 1\\
A-1090 Vienna, Austria}
\parbox[t]{0.4\textwidth}{\footnotesize
\hspace*{-1ex}$^2$Johann Radon Institute for Computational\\
\hspace*{1em}and Applied Mathematics (RICAM)\\
Altenbergerstra{\ss}e 69\\
A-4040 Linz, Austria}
\vspace*{2em}

\begin{abstract}
We analyse the ill-posedness of the photoacoustic imaging problem in the case of an attenuating medium. 
To this end, we introduce an attenuated photoacoustic operator and determine the asymptotic behaviour 
of its singular values. Dividing the known attenuation models into strong and weak attenuation classes, 
we show that for strong attenuation, the singular values of the attenuated photoacoustic operator decay 
exponentially, and in the weak attenuation case the singular values of the attenuated photoacoustic operator 
decay with the same rate as the singular values of the non-attenuated photoacoustic operator.
\end{abstract}

\section{Introduction}
In standard photoacoustic imaging, see e.g.~\cite{Wan09}, it is assumed that the medium is \emph{non-attenuating}, and
the imaging problem consists in visualising the spatially, compactly supported \emph{absorption density function}
$h:\R^3 \to \R$, appearing as a source term in the wave equation
\begin{equation} \label{eqWaveEquation}
\begin{aligned}
\partial_{tt}p(t,x)-\Delta p(t,x)&=\delta'(t)h(x),\quad &&t\in\R,\;x \in \R^3, \\
p(t,x)&=0,\quad&&t<0,\;x\in\R^3,
\end{aligned}
\end{equation}
from measurements $m(t,x)$ of the pressure $p$ for $(t,x) \in (0,\infty) \times \partial \Omega$, where $\partial \Omega$
is the boundary of a compact, convex set $\Omega$ containing the support of $h$.

In this paper, we consider photoacoustic imaging in attenuating media, where the propagation of the waves is 
described by the attenuated wave equation
\begin{equation}\label{eqWaveEquationAttenuation}
\begin{aligned}
\mathcal A_\kappa p(t,x)-\Delta p(t,x)&=\delta'(t)h(x),\quad &&t\in\R,\;x\in\R^3, \\
p(t,x)&=0,\quad&&t<0,\;x\in\R^3,
\end{aligned}
\end{equation}
where $\mathcal A_\kappa$ is the pseudo-differential operator defined in frequency domain by:
\begin{equation}\label{eqAttenuationOp}
\check{\mathcal A}_\kappa p(\omega,x) = -\kappa^2(\omega)\check p(\omega,x),\quad \omega \in \R,\;x \in \R^3,
\end{equation}
for some attenuation coefficient $\kappa:\R\to\C$ which admits a solution of \autoref{eqWaveEquationAttenuation}. 
Here $\check f$ denotes the one-dimensional inverse Fourier transform of $f$ with respect to time $t$, that is, for $f\in L^1(\R)$:
\[ \check f(\omega) = \frac1{\sqrt{2\pi}}\int_{-\infty}^\infty f(t)\e^{\i\omega t}\d t. \]
The \emph{attenuated photoacoustic imaging} problem consists in estimating $h$ from
measurements $m$ of $p$ on $\partial \Omega$ over time. 
The formal difference between \autoref{eqWaveEquationAttenuation} and \autoref{eqWaveEquation} is that the second time 
derivative operator $\partial_{tt}$ is replaced by a pseudo-differential operator $\mathcal A_\kappa$. 
We emphasise that standard photoacoustic imaging corresponds to $\kappa^2(\omega) = \omega^2$.

We review below, see \autoref{eqSolAttWave}, that in frequency domain the solution of \autoref{eqWaveEquationAttenuation} 
is  given by 
\[ \check p(\omega;x) = -\int_{\R^3} \frac{\i \omega}{4\pi\sqrt{2\pi}}\frac{\e^{\i\kappa(\omega)|x-y|}}{|x-y|}h(y)\d y. \]
We associate with this solution the \emph{time-integrated} photoacoustic operator in frequency domain:
\[ \PO h(\omega,x)= \frac{1}{4\pi\sqrt{2\pi}}\int_{\R^3} \frac{\e^{\i\kappa(\omega)|x-y|}}{|x-y|}h(y)\d y. \]

One goal of this paper is to characterise the degree of ill-posedness of the problem of inverting the time-integrated 
photoacoustic operator by estimating the decay rate of its singular values.
We mention however, that although the attenuated photoacoustic operator, giving the solution $\check p$, is related to the 
integrated photoacoustic operator by just time-differentiation, the singular values and functions 
of the photoacoustic operator have not been characterized so far. 

In this paper, we are identifying two classes of attenuation models (classes of functions $\kappa$), which 
correspond to \emph{weakly} and \emph{strongly} attenuating media. 
We prove that for weakly attenuating media the singular values $((\lambda_n(\PO^*\PO))^{\frac12})_{n=1}^\infty$ decay equivalently to  $n^{-\frac13}$, 
as in the standard photoacoustic imaging case, where this result has been proven in \cite{Pal10}. 
For the strongly attenuating models, the singular values are decaying exponentially, which is proven by using that in this case the operator $\PO^*\PO$ is an integral operator with smooth kernel.

\section{The Attenuated Wave Equation}

To model the wave propagation in an attenuated medium, we imitate the wave equation for the electric field $E:\R\times\R^3\to\R^3$ in an isotropic linear dielectric medium described by the electric susceptibility $\chi:\R\to\R$ (extended by $\chi(t)=0$ for $t<0$ to negative times):
\[ \frac1{c^2}\partial_{tt}E(t,x)+\frac1{c^2}\int_0^\infty\frac{\chi(\tau)}{\sqrt{2\pi}}\partial_{tt}E(t-\tau,x)\d\tau-\Delta E(t,x) = 0, \]
or written in terms of the inverse Fourier transforms $\check E$ and $\check\chi$ with respect to the time:
\begin{equation}\label{eqElectrodynamicModel}
-\frac{\omega^2}{c^2}\left(1+\check\chi(\omega)\right)\check E(\omega,x)-\Delta\check E(\omega,x) = 0.
\end{equation}

Analogously, we want to incorporate attenuation by replacing the second time derivatives in our equation \autoref{eqWaveEquation} by a pseudo-differential operator $\mathcal A_\kappa$ of the form \autoref{eqAttenuationOp} for some function $\kappa:\R\to\C$ (corresponding to $\frac\omega c\sqrt{1+\check\chi(\omega)}$ in the electrodynamic model).

We will interpret the equation \autoref{eqWaveEquationAttenuation} as an equation in the space of tempered distributions $\mathcal S'(\R\times\R^3)$ so that the Fourier transform and the $\delta$-distribution are both well-defined. To make sense of $\mathcal A_\kappa$ as an operator on $\mathcal S'(\R\times\R^3)$ and to be able to find a solution of \autoref{eqWaveEquationAttenuation}, we impose the following conditions on the function $\kappa$.

\begin{definition}\label{deAttCoeff}
We call a non-zero function $\kappa\in C^\infty(\R;\overline\H)$, where $\H=\{z\in\C\mid\Im z>0\}$ denotes the upper half complex plane and $\overline\H$ its closure in $\C$, an \emph{attenuation coefficient} if
\begin{enumerate}
\item\label{enAttCoeffPolBdd}
all the derivatives of $\kappa$ are polynomially bounded. That is, for every $\ell\in\N_0$ there exist constants $\kappa_1>0$ and $N\in\N$ such that
\begin{equation}\label{eqAttCoeffPolBdd}
|\kappa^{(\ell)}(\omega)| \le \kappa_1(1+|\omega|)^N,
\end{equation}
\item\label{enAttCoeffHol}
there exists a holomorphic continuation $\tilde\kappa:\overline\H\to\overline\H$ of $\kappa$ on the upper half plane, that is, $\tilde\kappa\in C(\overline\H;\overline\H)$ with $\tilde\kappa|_{\R}=\kappa$ and $\tilde\kappa:\H\to\overline\H$ is holomorphic; with
\[ |\tilde\kappa(z)|\le \tilde\kappa_1(1+|z|)^{\tilde N}\quad\text{for all}\quad z\in\overline\H \]
for some constants $\tilde\kappa_1>0$ and $\tilde N\in\N$.
\item\label{enAttCoeffSymm}
we have the symmetry $\kappa(-\omega)=-\overline{\kappa(\omega)}$ for all $\omega\in\R$.
\end{enumerate}
\end{definition}

The condition \ref{enAttCoeffPolBdd} in \autoref{deAttCoeff} ensures that the product $\kappa^2u$ of $\kappa^2$ with an arbitrary tempered distribution $u\in\mathcal S'(\R)$ is again in $\mathcal S'(\R)$ and therefore, the operator $\mathcal A_\kappa$ is well-defined.

\begin{definition}
Let $\kappa\in C^\infty(\R)$ be an attenuation coefficient. Then, we define the \emph{attenuation operator} $\mathcal A_\kappa:\mathcal S'(\R\times\R^3)\to\mathcal S'(\R\times\R^3)$ by its action on the tensor products $\phi\otimes\psi\in\mathcal S(\R\times\R^3)$, given by $(\phi\otimes\psi)(\omega,x)=\phi(\omega)\psi(x)$:
\begin{equation}\label{eqAttOp}
\left<\mathcal A_\kappa u,\phi\otimes\psi\right>_{\mathcal S',\mathcal S} = -\left<u,(\mathcal F^{-1}\kappa^2\mathcal F\phi)\otimes\psi\right>_{\mathcal S',\mathcal S},
\end{equation}
where $\mathcal F:\mathcal S(\R)\to\mathcal S(\R)$, $\mathcal F\phi(\omega)=\frac1{\sqrt{2\pi}}\int_{-\infty}^\infty\phi(t)\e^{-\i\omega t}\d t$ denotes the Fourier transform. This uniquely defines the operator $\mathcal A_\kappa$, see for example \cite[Lemma 6.2]{Tar07}.
\end{definition}
\begin{remark}
We use $\mathcal F$ when we are talking of the Fourier transform as an operator and use in the calculations $\hat\phi=\mathcal F\phi$ and $\check\phi=\mathcal F^{-1}\phi$.
We will use the notation $\mathcal F$ also for the Fourier transform on different spaces (in particular for the three-dimensional Fourier transform on $\mathcal S(\R^3)$).
\end{remark}

The condition \ref{enAttCoeffHol} in \autoref{deAttCoeff} is motivated by the fact that the function $\check\chi$ in the electrodynamic model \autoref{eqElectrodynamicModel} is the inverse Fourier transform of a function whose support is inside $[0,\infty)$ and can therefore be holomorphically extended to the upper half plane. We will see later, see \autoref{thSolAttWave}, that this condition guarantees that the attenuated wave equation \autoref{eqWaveEquationAttenuation} has a causal solution in $\mathcal S'(\R\times\R^3)$, that is a solution whose support is contained in $[0,\infty)\times\R^3$.

Finally, the condition \ref{enAttCoeffSymm} in \autoref{deAttCoeff} is required so that the attenuation operator $\mathcal A_\kappa$ maps real-valued distributions  to real-valued distributions: To see this, let $u\in\mathcal S'(\R\times\R^3)$ be a real-valued distribution. Then, for two real-valued functions $\phi\in\mathcal S(\R)$ and $\psi\in\mathcal S(\R^3)$, the relation~\autoref{eqAttOp} implies
\[ \overline{\left<\mathcal A_\kappa u,\phi\otimes\psi\right>_{\mathcal S',\mathcal S}} = -\left<u,\overline{\mathcal F^{-1}\kappa^2\mathcal F\phi}\otimes\psi\right>_{\mathcal S',\mathcal S}. \]
By substituting the variable $\omega$ by $-\omega$ in the Fourier integral below, we get that
\[ \overline{(\mathcal F^{-1}\kappa^2\mathcal F\phi)(t)} = \frac1{2\pi}\int_{-\infty}^\infty\e^{-\i\omega t}\overline{\kappa^2(\omega)}\int_{-\infty}^\infty\e^{\i\omega\tau}\phi(\tau)\d\tau\d\omega = (\mathcal F^{-1}\kappa_{\mathrm r}^2\mathcal F\phi)(t) \]
with $\kappa_{\mathrm r}$ given by $\kappa_{\mathrm r}(\omega)=\overline{\kappa(-\omega)}$. Thus, the condition $\overline{\left<\mathcal A_\kappa u,\phi\otimes\psi\right>_{\mathcal S',\mathcal S}}=\left<\mathcal A_\kappa u,\phi\otimes\psi\right>_{\mathcal S',\mathcal S}$ is equivalent to $\kappa^2=\kappa_{\mathrm r}^2$. Besides the case of a constant, real function $\kappa$ (something we are not interested in), this is equivalent to $\kappa=-\kappa_{\mathrm r}$ because of the condition $\Im\tilde\kappa(z)\ge0$ for all $z\in\overline\H$.

\subsection{Solution of the Attenuated Wave Equation}
In this section, we want to determine the solution $p\in\mathcal S'(\R\times\R^3)$ of \autoref{eqWaveEquationAttenuation}. To this end, we do a Fourier transform of the wave equation and end up with a Helmholtz equation for each value $\omega\in\R$, which in the case $\Im\kappa(\omega)>0$ has a unique solution in the space of tempered distributions.

\begin{lemma}\label{thHelmholtz}
Let $\kappa$ be a complex number with positive imaginary part, that is $\kappa \in \H$, $f\in L^2(\R^3)$ with compact essential support. Then, the Helmholtz equation
\[ \kappa^2\left<u,\phi\right>_{\mathcal S',\mathcal S}+\left<u,\Delta\phi\right>_{\mathcal S',\mathcal S}=\int_{\R^3}f(x)\phi(x)\d x,\quad\phi\in\mathcal S(\R^3), \]
has a unique solution $u\in\mathcal S'(\R^3)$, which is explicitly given by
\begin{equation}\label{eqHelmholtzSol}
\left<u,\phi\right>_{\mathcal S',\mathcal S} = -\frac1{4\pi}\int_{\R^3}\int_{\R^3}\frac{\e^{\i\kappa|x-y|}}{|x-y|}f(y)\d y\,\phi(x)\d x,\quad\phi\in\mathcal S(\R^3).
\end{equation}
\end{lemma}

\begin{proof}
Writing $\phi=\mathcal F^{-1}\hat\phi$, where $\mathcal F:\mathcal S(\R^3)\to\mathcal S(\R^3)$ denotes the three-dimensional Fourier transform,
we find with the function $\psi\in\mathcal S(\R^3)$ defined by $\psi=\kappa^2\phi+\Delta\phi$, and therefore $\hat\psi(k)=\mathcal F\psi(k)=(\kappa^2-|k|^2)\hat\phi(k)$, that
\begin{equation}\label{eqHelmholtzFourier}
\left<u,\psi\right>_{\mathcal S',\mathcal S} =\int_{\R^3}f(x)\phi(x)\d x = \frac1{(2\pi)^{\frac32}}\int_{\R^3}f(x)\int_{\R^3}\frac{\hat\psi(k)}{\kappa^2-|k|^2}\e^{\i\left<k,x\right>}\d k\d x.
\end{equation}
The inner integral is the inverse Fourier transform of a product and can thus be written as the convolution of two inverse Fourier transforms:
\begin{equation}\label{eqHelmholtzFourierTestFct}
\frac1{(2\pi)^{\frac32}}\int_{\R^3}\frac{\hat\psi(k)}{\kappa^2-|k|^2}\e^{\i\left<k,x\right>}\d k = \frac1{(2\pi)^3}\int_{\R^3}\psi(x-y)\int_{\R^3}\frac{\e^{\i\left<k,y\right>}}{\kappa^2-|k|^2}\d k\d y.
\end{equation}
Using spherical coordinates, we obtain by substituting $\rho=|k|$ and $\cos\theta=\frac{\left<k,y\right>}{|k||y|}$ that
\begin{align*}
\int_{\R^3}\frac{\e^{\i\left<k,y\right>}}{\kappa^2-|k|^2}\d k &= 2\pi\int_0^\infty\int_0^\pi\frac{\e^{\i\rho|y|\cos\theta}}{\kappa^2-\rho^2}\rho^2\sin\theta\d\theta\d\rho \\
&= \frac{4\pi}{|y|}\int_0^\infty\frac{\rho\sin(\rho|y|)}{\kappa^2-\rho^2}\d\rho = -\frac{2\pi\i}{|y|}\int_{-\infty}^\infty\frac{\rho\e^{\i\rho|y|}}{\kappa^2-\rho^2}\d\rho.
\end{align*}
Extending the integrand on the right hand side to a meromorphic function on the upper half complex plane, we can use the residue theorem to calculate the integral and find by taking into account that $\kappa \in \H$ that
\begin{equation}\label{eqHelmholtzFundSol}
\int_{\R^3}\frac{\e^{\i\left<k,y\right>}}{\kappa^2-|k|^2}\d k = -2\pi^2\frac{\e^{\i\kappa|y|}}{|y|}.
\end{equation}
Inserting \autoref{eqHelmholtzFundSol} into \autoref{eqHelmholtzFourierTestFct} and further into \autoref{eqHelmholtzFourier}, and remarking that $\psi$ is indeed an arbitrary function in $\mathcal S(\R^3)$, we end up with \autoref{eqHelmholtzSol}.
\end{proof}

To translate the initial condition in \autoref{eqWaveEquationAttenuation} that the solution $p$ vanishes for negative times into Fourier space, we use that the Fourier transform of such a function can be characterised by being polynomially bounded on the upper half complex plane away from the real axis.

We will briefly summarise the theory as we need it. For a detailed exposition, we refer to~\cite[Chapter 7.4]{Hoe03}.

\begin{definition}\label{deFourierLaplace}
Let $u\in\mathcal D'(\R)$ be a distribution with $\supp u\subset[0,\infty)$ such that $e_{-\eta}u\in\mathcal S'(\R)$ for every $\eta>0$, where we denote by $e_z\in C^\infty(\R)$, $z\in\C$, the function $e_z(t)=\e^{z t}$.

We define the \emph{adjoint Fourier--Laplace transform} $\check u:\H\to\C$ of $u$ by choosing for every point $z\in\H$ an arbitrary $\eta_z\in(0,\Im z)$ and by setting
\[ \check u(z) = \frac1{\sqrt{2\pi}}\left<e_{-\eta_z}u,\tilde e_{\i z+\eta_z}\right>_{\mathcal S',\mathcal S}. \]
Here $\tilde e_z$ denotes for every $z\in\C$ with $\Re z<0$ an arbitrary extension of the function $e_z|_{[0,\infty)}$ to the negative axis such that $\tilde e_z\in\mathcal S(\R)$. (The definition does not depend on the choice of the extension, since $\supp u\subset[0,\infty)$, see the proof of \cite[Theorem~2.3.3]{Hoe03}.)

Note that if $u$ is a regular distribution: $\left<u,\phi\right>_{\mathcal D',\mathcal D}=\int_{-\infty}^\infty U(t)\phi(t)\d t$ for all $\phi\in C^\infty_{\mathrm c}(\R)$ with an integrable function $U:\R\to\C$ with $\supp U\subset[0,\infty)$, then this is exactly the holomorphic extension of the inverse Fourier transform of $U$ to the upper half plane:
\[ \check u(z) = \frac1{\sqrt{2\pi}}\int_0^\infty U(t) e^{\i zt}\d t,\quad z\in\overline\H. \]
\end{definition}

With this construction, we have that the inverse Fourier transform of the tempered distribution $u_\eta=e_{-\eta}u$ for $\eta>0$ is the regular distribution corresponding to $\check u(\cdot+\i\eta)$, that is,
\begin{equation}\label{eqFourierLaplaceRelation}
\left<\mathcal F^{-1}u_\eta,\phi\right>_{\mathcal S',\mathcal S} = \int_{-\infty}^\infty\check u(\omega+\i\eta)\phi(\omega)\d\omega.
\end{equation}

Now, the causality of a distribution $u$, that is, $\supp u\subset[0,\infty)$, can be written in the form of a polynomial bound on its Fourier--Laplace transform $\check u$.

\begin{lemma}\label{thCausalityFourier}
We use again for every $z\in\C$ the notation $e_z\in C^\infty(\R)$ for the function $e_z(t)=\e^{\i zt}$.
\begin{enumerate}
\item
Let $u\in\mathcal D'(\R)$ be a distribution with $\supp u\subset[0,\infty)$ and such that $e_{-\eta}u\in\mathcal S'(\R)$ for every $\eta>0$.

Then, we find for every $\eta_1>0$ constants $C>0$ and $N\in\N$ such that the adjoint Fourier--Laplace transform $\check u$ of $u$ fulfils
\begin{equation}\label{eqCausalityFourier}
|\check u(z)|\le C(1+|z|)^N\quad\text{for all}\quad z\in\C\quad\text{with}\quad\Im z\ge\eta_1.
\end{equation}
\item
Conversely, if we have a holomorphic function $\check u:\H\to\C$ such that there exist for every $\eta_1>0$ constants $C>0$ and $N\in\N$ with
\begin{equation}\label{eqCausalityFourierReverse}
|\check u(z)|\le C(1+|z|)^N\quad\text{for all}\quad z\in\C\quad\text{with}\quad\Im z\ge\eta_1,
\end{equation}
then $\check u$ coincides with the adjoint Fourier--Laplace transform of a distribution $u\in\mathcal D'(\R)$ with $\supp u\subset[0,\infty)$ and the property that $e_{-\eta}u\in\mathcal S'(\R)$ for all $\eta>0$.
\end{enumerate}
\end{lemma}

\begin{proof}
\begin{enumerate}
\item[]
\item
Let $\eta_1>0$ and $\eta_0\in(0,\eta_1)$ be arbitrary. We choose a function $\psi\in C^\infty(\R)$ with $\psi(t)=1$ for $t\in(-\infty,0]$ and $\psi(t)=0$ for $t\in[1,\infty)$. Then, we write the given distribution $u$ in the form $u=u_1+u_2$ by setting
\[ \left<u_1,\phi\right>_{\mathcal D',\mathcal D}=\left<u,\psi\phi\right>_{\mathcal D',\mathcal D}\quad\text{and}\quad\left<u_2,\phi\right>_{\mathcal D',\mathcal D}=\left<u,(1-\psi)\phi\right>_{\mathcal D',\mathcal D}\quad\text{for all}\quad\phi\in C^\infty_{\mathrm c}(\R). \]

Since we have by assumption $e_{-\eta_0}u\in\mathcal S'(\R)$ and since $u_1$ has by construction compact support, we get that $e_{-\eta_0}u_2\in\mathcal S'(\R)$. Thus, because of $\supp(1-\psi)\subset[0,\infty)$, there exist constants $A_2>0$ and $N_2\in\N$ such that
\begin{equation}\label{eqFourierDistCont}
\left|\left<e_{-\eta_0}u_2,\phi\right>_{\mathcal S',\mathcal S}\right| = \left|\left<e_{-\eta_0}u,(1-\psi)\phi\right>_{\mathcal S',\mathcal S}\right| \le A_2\sum_{k,\ell=0}^{N_2}\sup_{t\in[0,\infty)}|t^\ell\phi^{(k)}(t)|\quad\text{for all}\quad\phi\in\mathcal S(\R).
\end{equation}

Moreover, since $e_{-\eta_0}u_1$ has compact support $\supp(e_{-\eta_0}u_1)\subset[0,1]$, we find, see for example \cite[Theorem 2.3.10]{Hoe03}, constants $A_1>0$ and $N_1\in\N$ so that
\begin{equation}\label{eqFourierCpDist}
\left|\left<e_{-\eta_0}u_1,\phi\right>_{\mathcal E',\mathcal E}\right| \le A_1\sum_{k=0}^{N_1}\sup_{t\in[0,1]}|\phi^{(k)}(t)|\quad\text{for all}\quad\phi\in C^\infty(\R).
\end{equation}

We now define as in \autoref{deFourierLaplace} for $z\in\C$ with $\Re z>0$ an extension $\tilde e_z\in\mathcal S(\R)$ of the function $e_z|_{[0,\infty)}$ and choose for every $z\in\H$ the function $\phi=\tilde e_{\i z+\eta_0}$ in \autoref{eqFourierDistCont} and \autoref{eqFourierCpDist}.
Then, there exists a constant $C>0$ such that with $N=\max\{N_1,N_2\}$
\begin{align*}
|\check u(z)| &= \frac1{\sqrt{2\pi}}\left|\left<e_{-\eta_0}u,\tilde e_{\i z+\eta_0}\right>_{\mathcal S',\mathcal S}\right| \\
&\le \frac1{\sqrt{2\pi}}\left|\left<e_{-\eta_0}u_1,\tilde e_{\i z+\eta_0}\right>_{\mathcal S',\mathcal S}\right|+\frac1{\sqrt{2\pi}}\left|\left<e_{-\eta_0}u_2,\tilde e_{\i z+\eta_0}\right>_{\mathcal S',\mathcal S}\right| \le C(1+|z|)^N
\end{align*}
holds for every $z\in\H$ with $\Im z\ge\eta_1$.
\item
To construct the distribution $u$, we define from the given function $\check u$ for every $\eta>0$ the distribution $u_\eta\in\mathcal S'(\R)$ via the relation \autoref{eqFourierLaplaceRelation}, so that the inverse Fourier transform of $u_\eta$ is given by the regular distribution corresponding to the function $\omega\mapsto\check u(\omega+\i\eta)$.

Now, we want to show that $e_\eta u_\eta$ is in fact independent of $\eta$, so that there exists a distribution $u$ such that $u_\eta=e_{-\eta}u$ for every $\eta>0$. To do so, we first remark that the derivative $\partial_\eta u_\eta$ of $u_\eta$ with respect to $\eta$ fulfils for every $\phi\in\mathcal S(\R)$
\[ \left<\partial_\eta u_\eta,\phi\right>_{\mathcal S',\mathcal S} = \i\int_{-\infty}^\infty\check u'(\omega+\i\eta)\hat\phi(\omega)\d\omega = -\i\int_{-\infty}^\infty\check u(\omega+\i\eta)\hat\phi'(\omega)\d\omega = \left<u_\eta,f\phi\right>_{\mathcal S',\mathcal S}, \]
where $\hat\phi=\mathcal F\phi$ denotes the Fourier transform of $\phi$ and $f(t)=-t$, so that $\hat\phi'=\i\mathcal F(f\phi)$. Thus, $\partial_\eta u_\eta=fu_\eta$ and therefore, $\partial_\eta(e_\eta u_\eta)=e_\eta(\partial_\eta u_\eta-fu_\eta)=0$, proving that $e_\eta u_\eta$ is independent of $\eta$.

So, the distribution $u=e_\eta u_\eta\in\mathcal D'(\R)$ is well-defined and fulfils by construction that $e_{-\eta}u\in\mathcal S'(\R)$ for every $\eta>0$.

Next, we want to show that $\supp u\subset[0,\infty)$. Let $\phi\in C^\infty_{\mathrm c}(\R)$ and write again $\hat\phi=\mathcal F\phi$. Then, by our construction, we have for every $\eta_1>0$ that
\begin{equation}\label{eqCausalityFourierSuppProof}
\left<e_{-\eta_1}u,\phi\right>_{\mathcal S',\mathcal S} = \int_{-\infty}^\infty\check u(\omega+\i\eta_1)\hat\phi(\omega)\d\omega.
\end{equation}
Since $\hat\phi$ is the Fourier transform of a function with compact support, we can extend it holomorphically to $\C$ and get for every $N_1\in\N_0$ a constant $C_1>0$ such that the upper bound
\[ |\hat\phi(z)| = \frac1{\sqrt{2\pi}}\left|\int_{-\infty}^\infty\phi(t)\e^{-\i zt}\d t\right| \le \frac{C_1}{(1+|z|)^{N_1}}\e^{\sup_{t\in\supp\phi}(t\Im z)} \]
holds.

Therefore, we can shift the line of integration in \autoref{eqCausalityFourierSuppProof} by an arbitrary value $\eta>0$ upwards in the upper half plane and get with the upper bound \autoref{eqCausalityFourierReverse} that
\[ \left|\left<e_{-\eta_1}u,\phi\right>_{\mathcal S',\mathcal S}\right| = \left|\int_{-\infty}^\infty\check u(\omega+\i(\eta_1+\eta))\hat\phi(\omega+\i\eta)\d\omega\right| \le A \e^{\eta\sup_{t\in\supp\phi}t} \]
for some constant $A>0$. Choosing now $\phi$ such that $\supp\phi\subset(-\infty,0)$ and taking the limit $\eta\to\infty$, the right hand side tends to zero, showing that $\left<e_{-\eta_1}u,\phi\right>_{\mathcal S',\mathcal S}=0$ whenever $\supp\phi\subset(-\infty,0)$. Thus, $e_{-\eta_1}u$, and therefore also $u$, has only support on $[0,\infty)$.

Finally, we verify that the Fourier--Laplace transform of $u$ is given by $\check u$. Indeed, given any $\omega\in\R$ and $\eta>0$, we have by construction for every $\eta_1\in(0,\eta)$ and every extension $\tilde e_z\in\mathcal S(\R)$ of $e_z|_{[0,\infty)}$ for $z\in\C$ with $\Re z>0$ that
\[ \left<e_{-\eta_1}u,\tilde e_{\i(\omega+\i\eta)+\eta_1}\right>_{\mathcal S',\mathcal S}=\int_{-\infty}^\infty\check u(\omega_1+\i\eta_1)\mathcal F\tilde e_{\i\omega+(\eta_1-\eta)}(\omega_1)\d\omega_1. \]
Since $\supp u\subset[0,\infty)$, we know that this expression is independent of the concrete choice of the extension $\tilde e_z$. Moreover, both sides are independent of $\eta_1$. Thus, letting on the right hand side $\tilde e_z$ converge to $e_z$ and $\eta_1$ to $\eta$, $\mathcal F\tilde e_{\i\omega+(\eta_1-\eta)}$ will tend to $\sqrt{2\pi}$ times the $\delta$-distribution at $\omega$, and we therefore get
\[ \frac1{\sqrt{2\pi}}\left<e_{-\eta_1}u,\tilde e_{\i(\omega+\i\eta)+\eta_1}\right>_{\mathcal S',\mathcal S}=\check u(\omega+\i\eta). \]
\end{enumerate}
\end{proof}

We now return to the solution of the attenuated wave equation \autoref{eqWaveEquationAttenuation}.
\begin{proposition}\label{thSolAttWave}
Let $\kappa$ be an attenuation coefficient and $\mathcal A_\kappa:\mathcal S'(\R\times\R^3)\to\mathcal S'(\R\times\R^3)$ be the corresponding attenuation operator. Let further $h\in L^2(\R^3)$ with compact essential support.

Then, the attenuated wave equation
\begin{equation}\label{eqSolAttWaveEq}
\left<\mathcal A_\kappa p,\vartheta\right>_{\mathcal S',\mathcal S}+\left<\Delta p,\vartheta\right>_{\mathcal S',\mathcal S} = -\int_{\R^3}h(x)\partial_t\vartheta(0,x)\d x,\quad\vartheta\in\mathcal S(\R\times\R^3),
\end{equation}
where the Laplace operator $\Delta:\mathcal S'(\R\times\R^3)\to\mathcal S'(\R\times\R^3)$ is defined by
\[ \left<\Delta u,\phi\otimes\psi\right>_{\mathcal S',\mathcal S} = \left<u,\phi\otimes(\Delta\psi)\right>_{\mathcal S',\mathcal S}\quad\text{for all}\quad\phi\in\mathcal S(\R),\;\psi\in\mathcal S(\R^3), \]
has a unique solution $p\in\mathcal S'(\R\times\R^3)$ with $\supp p\subset[0,\infty)\times\R^3$.

Moreover, $p$ is of the form
\begin{equation}\label{eqSolAttWave}
\left<p,\phi\otimes\psi\right>_{\mathcal S',\mathcal S} =  \int_{-\infty}^\infty\int_{\R^3}\int_{\R^3}G_\kappa(\omega,x-y)h(y)\d y\,\hat\phi(\omega)\psi(x)\d x\d\omega,
\end{equation}
where $\hat\phi$ denotes the Fourier transform of $\phi$ and $G$ denotes the integral kernel
\begin{equation}\label{eqSolAttWaveKernel}
G_\kappa(\omega,x) = -\frac{\i\omega}{4\pi\sqrt{2\pi}}\frac{\e^{\i\kappa(\omega)|x|}}{|x|},\quad\omega\in\R,\;x\in\R^3\setminus\{0\}.
\end{equation}
\end{proposition}

\begin{proof}
Let $p\in\mathcal S'(\R\times\R^3)$ be a solution of \autoref{eqSolAttWaveEq} with $\supp p\subset[0,\infty)\times\R^3$. We evaluate the equation~\autoref{eqSolAttWaveEq} for $\vartheta=\phi\otimes\psi\in C^\infty_{\mathrm c}(\R\times\R^3)$ and write $\mathcal F\phi=\hat\phi$. It then follows that
\begin{equation}\label{eqSolAttWaveFourier}
-\left<p,\mathcal F^{-1}(\kappa^2\hat\phi)\otimes\psi\right>_{\mathcal S',\mathcal S}+\left<p,\mathcal F^{-1}\hat\phi\otimes\Delta\psi\right>_{\mathcal S',\mathcal S} = \phi'(0)\int_{\R^3}h(x)\psi(x)\d x.
\end{equation}
For arbitrary $z\in\H$, we define the adjoint Fourier--Laplace transform $\check p(z)\in\mathcal S'(\R^3)$ of $p$ by
\[ \left<\check p(z),\psi\right>_{\mathcal S',\mathcal S} = \frac1{\sqrt{2\pi}}\left<p,\tilde e_{\i z}\otimes\psi\right>_{\mathcal S',\mathcal S}, \]
where $\tilde e_z\in\mathcal S(\R)$, $z\in\C$ with $\Re z<0$, is an arbitrary extension of $\tilde e_z(t)=\e^{zt}$ for $t\ge0$, see \autoref{deFourierLaplace}. Then, $z\mapsto\left<\check p(z),\psi\right>$ is holomorphic in the upper half plane $\H$ and we have
\begin{align*}
\left<p,\mathcal F^{-1}(\kappa^2\hat\phi)\otimes\psi\right>_{\mathcal S',\mathcal S} &= \lim_{\xi\downarrow0}\left<p,\tilde e_{-\xi} \mathcal F^{-1}(\kappa^2\hat\phi)\otimes\psi\right>_{\mathcal S',\mathcal S} \\
&= \lim_{\xi\downarrow0}\int_{-\infty}^\infty\left<\check p(\omega+\i\xi),\psi\right>_{\mathcal S',\mathcal S}\kappa^2(\omega)\hat\phi(\omega)\d\omega.
\end{align*}

We replace $\kappa$ in the integrand now by its holomorphic extension $\tilde\kappa:\overline\H\to\overline\H$, see \ref{enAttCoeffHol} in \autoref{deAttCoeff}, and also extend the Fourier transform $\hat\phi$ of the compactly supported function $\phi$ holomorphically to $\C$. Since $z\mapsto\left<\check p(z),\psi\right>$ is the adjoint Fourier--Laplace transform of a distribution with support on $[0,\infty)$, it is polynomially bounded, see \autoref{thCausalityFourier}. Moreover, we have by \autoref{deAttCoeff} of the attenuation coefficient a polynomial bound on $\tilde\kappa$ and get therefore with the dominated convergence theorem that
\[ \left<p,\mathcal F^{-1}(\kappa^2\hat\phi)\otimes\psi\right>_{\mathcal S',\mathcal S} = \lim_{\xi\downarrow0}\lim_{\eta\downarrow0}\int_{-\infty}^\infty\left<\check p(\omega+\i(\xi+\eta)),\psi\right>_{\mathcal S',\mathcal S}\tilde\kappa^2(\omega+\i\eta)\hat\phi(\omega+\i\eta)\d\omega. \]
Since all functions in the integrand are holomorphic in the upper half plane the integral is independent of $\eta$ and we can therefore remove the limit with respect to $\eta$. Using again the dominated convergence theorem, we can evaluate now the limit with respect to $\xi$ and obtain for arbitrary $\eta>0$ the equality
\[ \left<p,\mathcal F^{-1}(\kappa^2\hat\phi)\otimes\psi\right>_{\mathcal S',\mathcal S} = \int_{-\infty}^\infty\left<\check p(\omega+\i\eta),\psi\right>_{\mathcal S',\mathcal S}\tilde\kappa^2(\omega+\i\eta)\hat\phi(\omega+\i\eta)\d\omega. \]

Inserting this into the equation \autoref{eqSolAttWaveFourier} and arguing in the same way for the two other terms therein, we see that $\check p(z)\in\mathcal S'(\R^3)$ solves for every $z\in\H$ the equation
\[ \tilde\kappa^2(z)\left<\check p(z),\psi\right>_{\mathcal S',\mathcal S}+\left<\check p(z),\Delta\psi\right>_{\mathcal S',\mathcal S} = \frac{\i z}{\sqrt{2\pi}}\int_{\R^3}h(x)\psi(x)\d x. \]
Thus, by \autoref{thHelmholtz}, we get for every $z\in\H$ with $\Im\tilde\kappa(z)>0$ that
\begin{equation}\label{eqSolAttWaveFourierSol}
\left<\check p(z),\psi\right>_{\mathcal S',\mathcal S} = -\frac{\i z}{4\pi\sqrt{2\pi}}\int_{\R^3}\int_{\R^3}\frac{\e^{\i\tilde\kappa(z)|x-y|}}{|x-y|}h(y)\d y\,\psi(x)\d x
\end{equation}
is the only solution. However, since $\tilde\kappa$ is holomorphic, its imaginary part cannot vanish in any open set unless $\tilde\kappa$ were a constant, real function which is excluded by the symmetry condition \ref{enAttCoeffSymm} in \autoref{deAttCoeff}. Therefore, we can uniquely extend the formula \autoref{eqSolAttWaveFourierSol} for $\left<\check p(z),\psi\right>_{\mathcal S',\mathcal S}$ by continuity to all $z\in\overline\H$.

It remains to verify that $\supp p\subset[0,\infty)\times\R^3$. To see this, we use that $\Im\tilde\kappa(z)\ge0$ for every $z\in\overline\H$ to estimate the integral in \autoref{eqSolAttWaveFourierSol} by
\[ \left|\left<\check p(z),\psi\right>_{\mathcal S',\mathcal S}\right| \le C|z| \]
with some constant $C>0$. Therefore, by \autoref{thCausalityFourier}, $\left<\check p(z),\psi\right>_{\mathcal S',\mathcal S}$ is the Fourier--Laplace transform of a distribution with support in $[0,\infty)$.
\end{proof}

\subsection{Finite Propagation Speed}
Seeing the equation \autoref{eqSolAttWaveEq} as a generalisation of the wave equation, it is natural to additionally impose that the solution propagates with finite speed.
\begin{definition}
We say that the solution $p\in\mathcal S'(\R\times\R^3)$ of the equation \autoref{eqSolAttWaveEq} \emph{propagates with finite speed} $c>0$ if
\[ \supp p\subset\{(t,x)\in\R\times\R^3\mid|x|\le ct+R\} \]
whenever $\supp h\subset B_R(0)$.
\end{definition}

We can give an explicit characterisation of the equations whose solutions propagate with finite speed in terms of the holomorphic extension $\tilde\kappa$ of the attenuation coefficient $\kappa$.
\begin{lemma}\label{thFiniteSpeedViaIm}
The solution $p$ of the attenuated wave equation \autoref{eqSolAttWaveEq} propagates with finite speed $c>0$ if the holomorphic extension $\tilde\kappa$ of the attenuation coefficient $\kappa$ fulfils
\[ \Im(\tilde\kappa(z)-\tfrac zc)\ge0\quad\text{for every}\quad z\in\overline\H. \]

Conversely, if there exists a sequence $(z_\ell)_{\ell=1}^\infty\subset\H$ with the properties that
\begin{itemize}
\item
there exists a parameter $\eta_1>0$ such that $\Im(z_\ell)\ge\eta_1$ for all $\ell\in\N$,
\item
we have $|z_\ell|\to\infty$ for $\ell\to\infty$, and
\item
there exists a parameter $\delta>0$ such that
\[ \Im(\tilde\kappa(z_\ell)-\tfrac{z_\ell}c)\le-\delta|z_\ell|\quad\text{for all}\quad\ell\in\N, \]
\end{itemize}
then $p$ propagates faster than with speed $c$.
\end{lemma}
\begin{proof}
Since the solution $p\in\mathcal S'(\R\times\R^3)$ is a regular distribution with respect to the second component, see \autoref{eqSolAttWave}, having finite propagation speed is equivalent to the condition that the distribution $p(x)\in\mathcal S'(\R)$, given by
\[ \left<p(x),\phi\right>_{\mathcal S',\mathcal S} = \int_{-\infty}^\infty\int_{\R^3}G_\kappa(\omega,x-y)h(y)\d y\,\hat\phi(\omega)\d\omega,\quad x\in\R^3, \]
has $\supp p(x)\subset[\frac1c(|x|-R),\infty)$. Letting $h$ tend to a three dimensional $\delta$-distribution, we see that the distribution $g(x)\in\mathcal S'(\R)$, defined by
\[ \left<g(x),\phi\right>_{\mathcal S',\mathcal S} = \int_{-\infty}^\infty G_\kappa(\omega,x)\hat\phi(\omega)\d\omega, \]
has to fulfil $\supp g(x)\subset[\frac{|x|}c,\infty)$. If we shift $g(x)$ now by $\frac{|x|}c$ via $\tau:\mathcal S(\R)\to\mathcal S(\R)$, $(\tau\phi)(t)=\phi(t+\frac{|x|}c)$, this means that the distribution $g_\tau(x)\in\mathcal S'(\R)$, given by
\[ \left<g_\tau(x),\phi\right>_{\mathcal S',\mathcal S} = \left<g(x),\tau^{-1}\phi\right>_{\mathcal S',\mathcal S} = \int_{-\infty}^\infty G_\kappa(\omega,x)\e^{-\i\frac{\omega|x|}c}\hat\phi(\omega)\d\omega \]
has to have $\supp g_\tau(x)\subset[0,\infty)$. Extending the function $\omega\mapsto G_\kappa(\omega,x)\e^{-\i\frac{\omega|x|}c}$ to the upper half plane using the explicit formula \autoref{eqSolAttWaveKernel} for $G_\kappa$, we obtain the adjoint Fourier--Laplace transform $z\mapsto\check g_\tau(z,x)$ of the distribution $g_\tau(x)$:
\[ \check g_\tau(z,x) = -\frac{\i z}{4\pi\sqrt{2\pi}}\,\frac{\e^{\i(\tilde\kappa(z)-\frac zc)|x|}}{|x|}, \]
see \autoref{eqFourierLaplaceRelation}. According to \autoref{thCausalityFourier}, we can therefore equivalently characterise a finite propagation speed in terms of a polynomial bound on the function $\check g_\tau(\cdot,x)$.

\begin{itemize}
\item
If $\Im(\kappa(z)-\frac zc)\ge0$ for every $z\in\overline\H$, then the adjoint Fourier--Laplace transform of $g_\tau(x)$ fulfils that for every $x\in\R^3$ there exists a constant $C>0$ such that
\[ |\check g_\tau(z,x)| \le C|z|. \]
Thus, the condition \autoref{eqCausalityFourierReverse} of \autoref{thCausalityFourier} is satisfied and therefore $\supp g_\tau(x)\subset[0,\infty)$, so that $p$ propagates with the finite speed $c>0$.
\item
On the other hand, if there exists a sequence $(z_\ell)_{\ell=1}^\infty\subset\H$ with $\Im(z_\ell)\ge\eta_1$ for some $\eta_1>0$, $|z_\ell|\to\infty$, and $\Im(\tilde\kappa(z_\ell)-\frac{z_\ell}c)\le-\delta|z_\ell|$ for some $\delta>0$, then
\[ |\check g_\tau(z_\ell,x)| \ge \frac{|z_\ell|}{4\pi|x|\sqrt{2\pi}}\e^{\delta|x||z_\ell|}, \]
so that condition \autoref{eqCausalityFourier} of \autoref{thCausalityFourier} is violated and therefore the support of $g_\tau(x)$ cannot be contained in $[0,\infty)$.
\end{itemize}
\end{proof}

\begin{proposition}
Let $\kappa$ be an attenuation coefficient with the holomorphic extension $\tilde\kappa:\overline\H\to\overline\H$. Then, the solution $p$ of the attenuated wave equation \autoref{eqSolAttWaveEq} propagates with finite speed if and only if
\[ \lim_{\omega\to\infty}\frac{\tilde\kappa(\i\omega)}{\i\omega}>0. \]
In this case, it propagates with the speed $c=\lim_{\omega\to\infty}\frac{\i\omega}{\tilde\kappa(\i\omega)}$.
\end{proposition}
\begin{proof}
We make use of the theory of Nevanlinna functions, see for example \cite[Chapter 3.1]{Akh65}. Similar to the Riesz--Herglotz formula, which characterises the functions mapping the unit circle to the upper half plane, we have that all holomorphic functions $\tilde\kappa:\H\to\overline\H$ have an integral representation of the form
\begin{equation}\label{eqNevanlinna}
\tilde\kappa(z) = Az+B+\int_{-\infty}^\infty\frac{1+z\nu}{\nu-z}\d\sigma(\nu),\quad z\in\H,
\end{equation}
where $\sigma:\R\to\R$ is a monotonically increasing function of bounded variation and $A\ge0$ and $B\in\R$ are arbitrary parameters, and vice versa, see \cite[Formula 3.3]{Akh65}.

Then, $\tilde\kappa(z)-Az$ is still of the form \autoref{eqNevanlinna} and therefore is a holomorphic function mapping $\H$ to $\overline\H$. In particular, it satisfies $\Im(\tilde\kappa(z)-Az)\ge0$ for all $z\in\H$. Thus, if $A>0$, $p$ propagates with the finite speed $c=\frac1A$ according to \autoref{thFiniteSpeedViaIm}.

Evaluating $\tilde\kappa$ along the imaginary axis, we find that asymptotically as $\omega\to\infty$
\begin{equation}\label{eqFinitePropagationAsymptotics}
\tilde\kappa(\i\omega) = \i\omega\left(A+\frac B{\i\omega}+\int_{-\infty}^\infty\frac{1+\i\omega\nu}{\i\omega(\nu-\i\omega)}\d\sigma(\nu)\right) = \i\omega(A+o(1)).
\end{equation}
Thus,
\[ A=\lim_{\omega\to\infty}\frac{\tilde\kappa(\i\omega)}{\i\omega}. \]
Moreover for $A=0$, we see from \autoref{eqFinitePropagationAsymptotics} that for every choice of $c>0$, we have the behaviour $\tilde\kappa(\i\omega)-\frac{\i\omega}c=\i\omega(-\frac1c+o(1))$ and therefore $\Im(\tilde\kappa(\i\omega)-\frac{\i\omega}c)\le-\frac\omega{2c}$ for all $\omega\ge\omega_0$ for a sufficiently large $\omega_0$. Thus, by \autoref{thFiniteSpeedViaIm}, $p$ cannot have finite propagation speed for $A=0$.
\end{proof}


\section{Examples of Attenuation Models}
The following examples of attenuation coefficient have been collected in \cite{KowSch12}, where also references to original
papers can be found. In this section, we review them and catalog them into two groups which are characterised by different spectral
behaviour.

If the attenuation in the medium increases faster than some power of the frequency, we are in the case of strong attenuation.
\begin{definition}\label{deStrongAttenuation}
We call an attenuation coefficient $\kappa\in C^\infty(\R;\overline\H)$ a \emph{strong attenation coefficient} if it fulfils that
\begin{equation}\label{eqStrongAttenuation}
\Im\kappa(\omega)\ge\kappa_0|\omega|^\beta\quad\text{for all}\quad\omega\in\R\quad\text{with}\quad|\omega|\ge\omega_0
\end{equation}
for some constants $\kappa_0>0$, $\beta>0$, and $\omega_0\ge0$.
\end{definition}

A common example, which has the drawback of an infinite propagation speed, is the thermo-viscous model, see \autoref{tbThermoViscous}. In \cite{KowSchBon11}, the authors modified this model to obtain one with finite propagation speed, see \autoref{tbKowSchBon}. Other models, trying to match the heuristic power law behaviour of the attenuation are the power law in \autoref{tbPowerLaw} and Szabo's model, see \autoref{tbSza}, where we chose the modified version introduced in~\cite{KowSch12} as the original one does not lead to a causal model.

In these tables and in the following we always use the principal branch of the complex roots, that is, we define for $\gamma\in\C$
\[ (r\e^{\i\varphi})^\gamma = \e^{\gamma(\log(r)+\i\varphi)}\quad\text{for every}\quad r>0,\;\varphi\in(-\pi,\pi). \]

\begin{remark}
The attenuation coefficients in \autoref{tbKowSchBon}, \autoref{tbPowerLaw}, and \autoref{tbSza} do not fulfil the smoothness assumption $\kappa\in C^\infty(\R)$. However, this requirement originates mainly from our choice of solution concept for the attenuated wave equation \autoref{eqWaveEquationAttenuation} and we may still consider formula \autoref{eqSolAttWaveEqSimplified} as definition of the solution~$p$ of \autoref{eqWaveEquationAttenuation} if $\kappa$ is non-smooth. In particular, the smoothness assumption is not required for the derivation of the decay of the singular values of the integrated photoacoustic operator.
\end{remark}

\renewcommand{\floatpagefraction}{0.8}
\def\arraystretch{2}

\begin{table}[tb]
\begin{center}
\begin{tabular}{|l|p{11cm}|}
\hline
Name: & Thermo-viscous model, see for example \cite[Chapter 8.2]{KinFreCopSan00} \\ \hline
Attenuation coefficient: & $\kappa:\R\to\C$, $\displaystyle\kappa(\omega) = \frac{\omega}{\sqrt{1 - \i\tau\omega }}$ \\
Parameters: & $\tau>0$ \\ \hline
Holomorpic extension: & $\tilde\kappa:\overline\H\to\C$, $\displaystyle\tilde\kappa(z) = \frac{z}{\sqrt{1 - \i \tau z}}$ \\
Upper bound: & $|\tilde\kappa(z)| \le |z|$ for all $z\in\overline\H$\newline\mbox{}\newline
	\footnotesize This follows from $|1-\i\tau z|\ge\Re(1-\i\tau z)\ge1$ for $z\in\overline\H$. \\ \hline
Propagation speed: & $\displaystyle c=\lim_{\omega\to\infty}\frac{\i\omega}{\tilde\kappa(\i\omega)}=\lim_{\omega\to\infty}\sqrt{1+\tau\omega}=\infty$ \\ \hline
Attenuation type: & Strong attenuation coefficient\newline\mbox{}\newline
\footnotesize
Indeed a Taylor expansion with respect to $\frac1\omega$ around $0$ yields for $\omega\to\infty$:
\[ \begin{aligned}
\Im\kappa(\omega)&=\Im\sqrt{\frac{\i\omega}\tau}\left(1+\frac\i{\tau\omega}\right)^{-\frac12} \\
&= \Im\sqrt{\frac{\i\omega}\tau}\left(1+\Ord(\omega^{-1})\right) =\sqrt{\frac\omega{2\tau}}+\Ord(\omega^{-1}).
\end{aligned} \]
\\ \hline
Range of $\tilde\kappa$: &
\raisebox{-0.9\height}{
\begin{tikzpicture}[font=\tiny]
\node[inner sep=0pt,anchor=south] at (0,0) {\includegraphics[width=80pt]{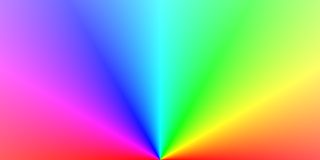}};
\draw[->](-45pt,0pt) -- (45pt,0pt);
\draw[->](0pt,-45pt) -- (0pt,45pt);
\draw (-40pt,-40pt) rectangle (40pt,40pt);
\draw(20pt,2pt) -- (20pt,-2pt) node[below] {$1$};
\draw(2pt,20pt) -- (-2pt,20pt) node[left] {$\i$};
\draw(-20pt,2pt) -- (-20pt,-2pt) node[below] {$-1$};
\draw(2pt,-20pt) -- (-2pt,-20pt) node[left] {$-\i$};

\draw[->,thick] (60pt,0) -- (80pt,0) node[pos=0.5,above] {$\tilde\kappa$};

\begin{scope}[shift={(140pt,0pt)}]
\node[inner sep=0pt,anchor=center] at (0,0) {\includegraphics[width=80pt]{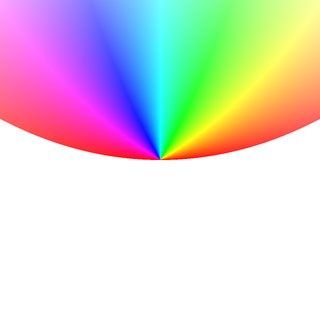}};
\draw[->](-45pt,0pt) -- (45pt,0pt);
\draw[->](0pt,-45pt) -- (0pt,45pt);
\draw (-40pt,-40pt) rectangle (40pt,40pt);
\draw(20pt,2pt) -- (20pt,-2pt) node[below] {$1$};
\draw(2pt,20pt) -- (-2pt,20pt) node[left] {$\i$};
\draw(-20pt,2pt) -- (-20pt,-2pt) node[below] {$-1$};
\draw(2pt,-20pt) -- (-2pt,-20pt) node[left] {$-\i$};
\end{scope}
\end{tikzpicture}}\newline\mbox{}\newline
\footnotesize
To see analytically that $\tilde\kappa$ maps the upper half plane $\overline\H$ into itself, we first remark that because of the symmetry $\tilde\kappa(-\bar z)=-\overline{\tilde\kappa(z)}$, it is enough to show that the first quadrant $Q_{++}=\{z\in\C\mid\Re(z)\ge0,\;\Im(z)\ge0\}$ is mapped under $\tilde\kappa$ into $\overline\H$.\newline
Since $f:\bar\C\to\bar\C$, $f(z)=\frac1{1-\i\tau z}$ is a Möbius transform which maps $Q_{++}$ to the half ball $\bar B_{\frac12}(\frac12)\cap Q_{++}$ and $\tilde\kappa$ is the composition $\tilde\kappa(z)=z\sqrt{f(z)}$, we indeed have $\tilde\kappa(Q_{++})\subset\overline\H$. \\ \hline
\end{tabular}
\end{center}
\caption{The thermo-viscous model.}
\label{tbThermoViscous}
\end{table}

\begin{table}[!ht]
\begin{center}
\begin{tabular}{|l|p{11cm}|} \hline
Model: & Kowar--Scherzer--Bonnefond model, see \cite{KowSchBon11} \\ \hline
Attenuation coefficient: & $\kappa:\R\to\C$, $\displaystyle\kappa(\omega) = \omega\left(1+\frac{\alpha}{\sqrt{1+(-\i\tau\omega)^\gamma}}\right)$ \\
Parameters: & $\gamma\in(0,1)$, $\alpha>0$, $\tau>0$ \\\hline
Holomorphic extension: & $\tilde\kappa:\overline\H\to\C$, $\displaystyle\tilde\kappa(z) = z\left(1+\frac{\alpha}{\sqrt{1+(-\i\tau z)^\gamma}}\right)$ \\
Upper bound: & $\displaystyle|\tilde\kappa(z)| \le (1+\alpha)|z|$ \newline\mbox{}\newline
	\footnotesize This follows from $|1+(-\i\tau z)^\gamma|\ge|\Re(1+(-\i\tau z)^\gamma)|\ge1$ for $z\in\overline\H$. \\ \hline
Propagation speed: & $\displaystyle c=\lim_{\omega\to\infty}\frac{\i\omega}{\tilde\kappa(\i\omega)}=\lim_{\omega\to\infty}\frac1{1+\frac\alpha{\sqrt{1+(\tau\omega)^\gamma}}}=1$ \\ \hline
Attenuation type: & Strong attenuation coefficient\newline\mbox{}\newline
	\footnotesize A Taylor expansion with respect to $\omega^{-\gamma}$ around $0$ yields for $\omega\to\infty$:
	\[ \begin{aligned}
	\Im\kappa(\omega) &= \alpha\omega\Im\big((-\i\tau\omega)^{-\frac\gamma2}(1+(-\i\tau\omega)^{-\gamma})^{-\frac12}\big) \\
	&= \alpha\omega\Im\big((-\i\tau\omega)^{-\frac\gamma2}+\Ord(\omega^{-\frac32\gamma})\big) \\
	&= \alpha\tau^{-\frac\gamma2}\sin(\tfrac{\pi\gamma}4)\omega^{1-\frac\gamma2}\big(1+\Ord(\omega^{-\gamma})\big).
	\end{aligned} \]
	\\ \hline
Range of $\tilde\kappa$: &
\raisebox{-0.9\height}{
\begin{tikzpicture}[font=\tiny]
\node[inner sep=0pt,anchor=south] at (0,0) {\includegraphics[width=80pt]{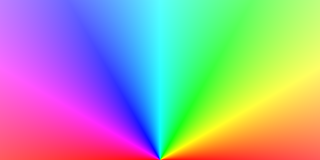}};
\draw[->](-45pt,0pt) -- (45pt,0pt);
\draw[->](0pt,-45pt) -- (0pt,45pt);
\draw (-40pt,-40pt) rectangle (40pt,40pt);
\draw(20pt,2pt) -- (20pt,-2pt) node[below] {$1$};
\draw(2pt,20pt) -- (-2pt,20pt) node[left] {$\i$};
\draw(-20pt,2pt) -- (-20pt,-2pt) node[below] {$-1$};
\draw(2pt,-20pt) -- (-2pt,-20pt) node[left] {$-\i$};

\draw[->,thick] (60pt,0) -- (80pt,0) node[pos=0.5,above] {$\tilde\kappa$};

\begin{scope}[shift={(140pt,0pt)}]
\node[inner sep=0pt,anchor=center] at (0,0) {\includegraphics[width=80pt]{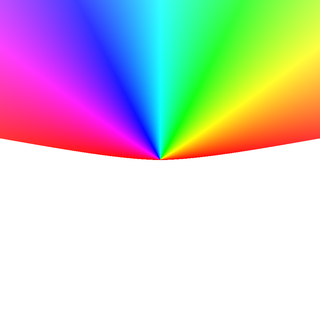}};
\draw[->](-45pt,0pt) -- (45pt,0pt);
\draw[->](0pt,-45pt) -- (0pt,45pt);
\draw (-40pt,-40pt) rectangle (40pt,40pt);
\draw(20pt,2pt) -- (20pt,-2pt) node[below] {$1$};
\draw(2pt,20pt) -- (-2pt,20pt) node[left] {$\i$};
\draw(-20pt,2pt) -- (-20pt,-2pt) node[below] {$-1$};
\draw(2pt,-20pt) -- (-2pt,-20pt) node[left] {$-\i$};
\end{scope}
\end{tikzpicture}}\newline\mbox{}\newline
\footnotesize To see where $\tilde\kappa$ maps the upper half plane, we write $\tilde\kappa$ in the form
\[ \tilde\kappa(z)=z(1+\alpha\sqrt{f_2(f_1(z))})\quad\text{with}\quad f_2(z)=\frac1{1+z},\;f_1(z)=(-\i\tau z)^\gamma. \]
Now, $f_1$ maps the first quadrant $Q_{++}$ in a subset of the fourth quadrant $Q_{+-}=\{z\in\C\mid\Re(z)\ge0,\Im(z)\le0\}$. And $f_2$ is a Möbius transform which maps $Q_{+-}$ to the half ball $\bar B_{\frac12}(\frac12)\cap Q_{++}$.

Thus, since the product of two points in the first quadrant $Q_{++}$ is in the upper half plane, $\tilde\kappa(Q_{++})\subset\overline\H$ and because of the symmetry $\tilde\kappa(-\bar z)=-\overline{\tilde\kappa(z)}$, we therefore have $\tilde\kappa(\overline\H)\subset\overline\H$.
\\\hline
\end{tabular}
\end{center}
\caption{The Kowar--Scherzer--Bonnefond model.}
\label{tbKowSchBon}
\end{table}

\begin{table}[!ht]
\begin{center}
\begin{tabular}{|l|p{11cm}|} \hline
Model: & Power law, see for example \cite{Sza94} \\ \hline
Attenuation coefficient: & $\kappa:\R\to\C$, $\kappa(\omega) = \omega+\i\alpha(-\i\omega)^\gamma$ \\
Parameters: & $\gamma\in(0,1)$, $\alpha>0$ \\\hline
Holomorphic extension: & $\tilde\kappa:\overline\H\to\C$, $\tilde\kappa(z) = z+\i\alpha(-\i z)^\gamma$ \\
Upper bound: & $|\tilde\kappa(z)| \le |z|+\alpha|z|^\gamma \le \alpha(1-\gamma)+(1+\alpha\gamma)|z|$ \newline\mbox{}\newline
	\footnotesize The second inequality uses Young's inequality to estimate $|z|^\gamma\le\gamma|z|+1-\gamma$. \\ \hline
Propagation speed: & $\displaystyle c=\lim_{\omega\to\infty}\frac{\i\omega}{\tilde\kappa(\i\omega)}=\lim_{\omega\to\infty}\frac1{1+\alpha\omega^{\gamma-1}}=1$ \\ \hline
Attenuation type: & Strong attenuation coefficient\newline\mbox{}\newline
	\footnotesize We have $\Im\kappa(\omega) = \alpha\sin\left((1-\gamma)\tfrac\pi2\right)|\omega|^\gamma$.
	\\ \hline
Range of $\tilde\kappa$: &
\raisebox{-0.9\height}{
\begin{tikzpicture}[font=\tiny]
\node[inner sep=0pt,anchor=south] at (0,0) {\includegraphics[width=80pt]{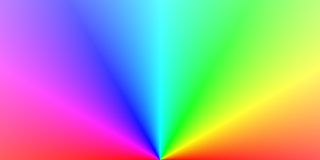}};
\draw[->](-45pt,0pt) -- (45pt,0pt);
\draw[->](0pt,-45pt) -- (0pt,45pt);
\draw (-40pt,-40pt) rectangle (40pt,40pt);
\draw(20pt,2pt) -- (20pt,-2pt) node[below] {$1$};
\draw(2pt,20pt) -- (-2pt,20pt) node[left] {$\i$};
\draw(-20pt,2pt) -- (-20pt,-2pt) node[below] {$-1$};
\draw(2pt,-20pt) -- (-2pt,-20pt) node[left] {$-\i$};

\draw[->,thick] (60pt,0) -- (80pt,0) node[pos=0.5,above] {$\tilde\kappa$};

\begin{scope}[shift={(140pt,0pt)}]
\node[inner sep=0pt,anchor=center] at (0,0) {\includegraphics[width=80pt]{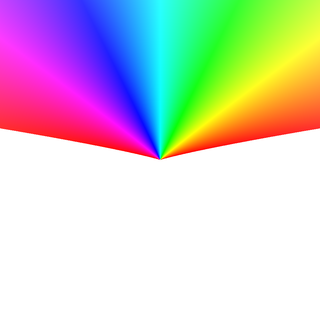}};
\draw[->](-45pt,0pt) -- (45pt,0pt);
\draw[->](0pt,-45pt) -- (0pt,45pt);
\draw (-40pt,-40pt) rectangle (40pt,40pt);
\draw(20pt,2pt) -- (20pt,-2pt) node[below] {$1$};
\draw(2pt,20pt) -- (-2pt,20pt) node[left] {$\i$};
\draw(-20pt,2pt) -- (-20pt,-2pt) node[below] {$-1$};
\draw(2pt,-20pt) -- (-2pt,-20pt) node[left] {$-\i$};
\end{scope}
\end{tikzpicture}}\newline\mbox{}\newline
\footnotesize That the range of $\tilde\kappa$ is a subset of $\overline\H$ follows immediately from
\[ \Im\tilde\kappa(r\e^{\i\varphi}) = r\sin\varphi+\alpha r^\gamma\sin\left(\frac\pi2+\left(\varphi-\frac\pi2\right)\gamma\right)\ge0 \]
for all $r\ge0$ and $\varphi\in[0,\pi]$. \\\hline
\end{tabular}
\end{center}
\caption{The power law model.}
\label{tbPowerLaw}
\end{table}

\begin{table}[!ht]
\begin{center}
\begin{tabular}{|l|p{11cm}|} \hline
Model: & Modified Szabo model, see \cite{KowSch12} and, for the original version, \cite{Sza94} \\ \hline
Attenuation coefficient: & $\kappa:\R\to\C$, $\displaystyle\kappa(\omega) = \omega\sqrt{1+\alpha(-\i\omega)^{\gamma-1}}$ \\
Parameters: & $\gamma\in(0,1)$, $\alpha>0$ \\\hline
Holomorphic extension: & $\tilde\kappa:\overline\H\to\C$, $\displaystyle\tilde\kappa(z) = z\sqrt{1+\alpha(-\i z)^{\gamma-1}}$ \\
Upper bound: & $|\tilde\kappa(z)| \le |z|+\sqrt\alpha|z|^{\frac{\gamma+1}2} \le \frac12\alpha(1-\gamma)+(1+\frac\alpha2(1+\gamma))|z|$ \newline\mbox{}\newline
	\footnotesize The second inequality uses Young's inequality to estimate $|z|^{\frac{1+\gamma}2}\le\frac12(1+\gamma)|z|+\frac12(1-\gamma)$. \\ \hline
Propagation speed: & $\displaystyle c=\lim_{\omega\to\infty}\frac{\i\omega}{\tilde\kappa(\i\omega)}=\lim_{\omega\to\infty}\frac1{\sqrt{1+\alpha\omega^{\gamma-1}}}=1$ \\ \hline
Attenuation type: & Strong attenuation coefficient\newline\mbox{}\newline
	\footnotesize A Taylor expansion with respect to $\omega^{\gamma-1}$ around $0$ yields for $\omega\to\infty$:
	\[ \Im\kappa(\omega) = \Im\left(\omega\left(1+2\alpha(-\i\omega)^{\gamma-1}\right)^{\frac12}\right) = \omega+\Ord(\omega^\gamma). \]
	\\ \hline
Range of $\tilde\kappa$: &
\raisebox{-0.9\height}{
\begin{tikzpicture}[font=\tiny]
\node[inner sep=0pt,anchor=south] at (0,0) {\includegraphics[width=80pt]{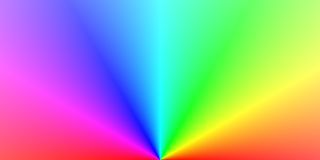}};
\draw[->](-45pt,0pt) -- (45pt,0pt);
\draw[->](0pt,-45pt) -- (0pt,45pt);
\draw (-40pt,-40pt) rectangle (40pt,40pt);
\draw(20pt,2pt) -- (20pt,-2pt) node[below] {$1$};
\draw(2pt,20pt) -- (-2pt,20pt) node[left] {$\i$};
\draw(-20pt,2pt) -- (-20pt,-2pt) node[below] {$-1$};
\draw(2pt,-20pt) -- (-2pt,-20pt) node[left] {$-\i$};

\draw[->,thick] (60pt,0) -- (80pt,0) node[pos=0.5,above] {$\tilde\kappa$};

\begin{scope}[shift={(140pt,0pt)}]
\node[inner sep=0pt,anchor=center] at (0,0) {\includegraphics[width=80pt]{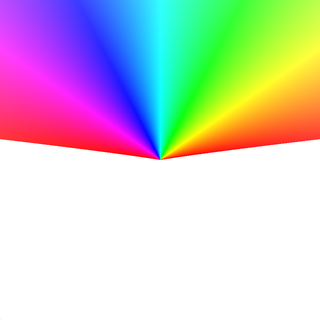}};
\draw[->](-45pt,0pt) -- (45pt,0pt);
\draw[->](0pt,-45pt) -- (0pt,45pt);
\draw (-40pt,-40pt) rectangle (40pt,40pt);
\draw(20pt,2pt) -- (20pt,-2pt) node[below] {$1$};
\draw(2pt,20pt) -- (-2pt,20pt) node[left] {$\i$};
\draw(-20pt,2pt) -- (-20pt,-2pt) node[below] {$-1$};
\draw(2pt,-20pt) -- (-2pt,-20pt) node[left] {$-\i$};
\end{scope}
\end{tikzpicture}}\newline\mbox{}\newline
\footnotesize To determine the range, we write $\tilde\kappa$ in the form
\[ \tilde\kappa(z) = z\sqrt{1+f(z)}\quad\text{with}\quad f(z)=\alpha(-\i z)^{\gamma-1}. \]
Now, since $\gamma-1<0$, $f$ maps the first quadrant $Q_{++}$ to a subset of $Q_{++}$. Thus, since the product of two points in the first quadrant is in the upper half plane, we have $\tilde\kappa(Q_{++})\subset\overline\H$ and because of the symmetry $\tilde\kappa(-\bar z)=-\overline{\tilde\kappa(z)}$ therefore $\tilde\kappa(\overline\H)\subset\overline\H$. \\\hline
\end{tabular}
\end{center}
\caption{The modified Szabo model.}
\label{tbSza}
\end{table}

In the case where the attenuation decreases sufficiently fast as the frequency increases, we call the medium weakly attenuating.
\begin{definition}\label{deWeakAttenuation}
We call an attenuation coefficient $\kappa\in C^\infty(\R;\overline\H)$ a \emph{weak attenuation coefficient} if it is of the form
\[ \kappa(\omega) = \frac\omega c+\i\kappa_\infty+\kappa_*(\omega),\quad\omega\in\R, \]
for some constants $c>0$ and $\kappa_\infty\ge0$ and a bounded function $\kappa_*\in C^\infty(\R)\cap L^2(\R)$.
\end{definition}

Clearly, the non-attenuating case where $\kappa(\omega)=\frac\omega c$ with $c>0$, so that the attenuated wave equation \autoref{eqWaveEquationAttenuation} reduces to the linear wave equation, falls under this category and we will later on treat the case of a weak attenuation coefficient as a perturbation of this non-attenuated case.

A non-trivial example of a weak attenuation model is the model by Nachman, Smith, and Waag, see \autoref{tbNacSmiWaa}.

\begin{table}[!ht]
\begin{center}
\begin{tabular}{|l|p{11cm}|}
\hline
Name: & Nachman--Smith--Waag model, see \cite{NacSmiWaa90} \\ \hline
Attenuation coefficient: & $\kappa:\R\to\C$, $\displaystyle\kappa(\omega) = \frac\omega{c_0}\sqrt{\frac{1-\i\tilde\tau\omega}{1-\i\tau\omega}}$ \\
Parameters: & $c_0>0$, $\tau>0$, $\tilde\tau\in(0,\tau)$ \\ \hline
Holomorpic extension: & $\tilde\kappa:\overline\H\to\C$, $\displaystyle\frac z{c_0}\sqrt{\frac{1-\i\tilde\tau z}{1-\i\tau z}}$ \\
Upper bound: & $|\tilde\kappa(z)| \le \frac1{c_0}|z|$ for all $z\in\overline\H$\newline\mbox{}\newline
	\footnotesize This follows from $|1-\i\tilde\tau z|\le|1-\i\tau z|$ for $z\in\overline\H$. \\ \hline
Propagation speed: & $\displaystyle c=\lim_{\omega\to\infty}\frac{\i\omega}{\tilde\kappa(\i\omega)}=\lim_{\omega\to\infty}c_0\sqrt{\frac{1+\tau\omega}{1+\tilde\tau\omega}}=c_0\sqrt{\frac\tau{\tilde\tau}}$ \\ \hline
Attenuation Type: & Weak attenuation coefficient \newline\mbox{}\newline
\footnotesize
Indeed a Taylor expansion with respect to $\frac1\omega$ around $0$ yields for $\omega\to\infty$:
\[ \begin{aligned}
\kappa(\omega)
&=\frac\omega{c_0}\sqrt{\frac{\tilde\tau}\tau}\left(1+\frac\i{\tilde\tau\omega}\right)^{\frac12}\left(1+\frac\i{\tau\omega}\right)^{-\frac12} \\
&= \frac\omega{c_0}\sqrt{\frac{\tilde\tau}\tau}\left(1+\frac\i{2\tilde\tau\omega}+\Ord(\omega^{-2})\right)\left(1-\frac\i{2\tau\omega}+\Ord(\omega^{-2})\right) \\
&= \frac\omega{c_0}\sqrt{\frac{\tilde\tau}\tau}+\frac{\i(\tau-\tilde\tau)}{2c_0\tau\sqrt{\tilde\tau\tau}}+\Ord(\omega^{-1}).
\end{aligned} \]
\\ \hline
Range of $\tilde\kappa$: &
\raisebox{-0.9\height}{
\begin{tikzpicture}[font=\tiny]
\node[inner sep=0pt,anchor=south] at (0,0) {\includegraphics[width=80pt]{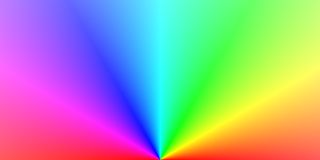}};
\draw[->](-45pt,0pt) -- (45pt,0pt);
\draw[->](0pt,-45pt) -- (0pt,45pt);
\draw (-40pt,-40pt) rectangle (40pt,40pt);
\draw(20pt,2pt) -- (20pt,-2pt) node[below] {$1$};
\draw(2pt,20pt) -- (-2pt,20pt) node[left] {$\i$};
\draw(-20pt,2pt) -- (-20pt,-2pt) node[below] {$-1$};
\draw(2pt,-20pt) -- (-2pt,-20pt) node[left] {$-\i$};

\draw[->,thick] (60pt,0) -- (80pt,0) node[pos=0.5,above] {$\tilde\kappa$};

\begin{scope}[shift={(140pt,0pt)}]
\node[inner sep=0pt,anchor=center] at (0,0) {\includegraphics[width=80pt]{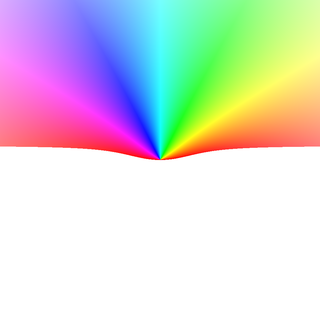}};
\draw[->](-45pt,0pt) -- (45pt,0pt);
\draw[->](0pt,-45pt) -- (0pt,45pt);
\draw (-40pt,-40pt) rectangle (40pt,40pt);
\draw(20pt,2pt) -- (20pt,-2pt) node[below] {$1$};
\draw(2pt,20pt) -- (-2pt,20pt) node[left] {$\i$};
\draw(-20pt,2pt) -- (-20pt,-2pt) node[below] {$-1$};
\draw(2pt,-20pt) -- (-2pt,-20pt) node[left] {$-\i$};
\end{scope}
\end{tikzpicture}}\newline\mbox{}\newline
\footnotesize The Möbius transform $f:\bar\C\to\bar\C$, $f(z)=\frac{1-\i\tilde\tau z}{1-\i\tau z}$ maps the first quadrant $Q_{++}$ to $\bar B_r(a)\cap Q_{++}$ where $r=\frac12(1-\frac{\tilde\tau}\tau)$ and $a=\frac12(1+\frac{\tilde\tau}\tau)$. Thus, $\tilde\kappa(z)=\frac z{c_0}\sqrt{f(z)}\in\overline\H$ for $z\in Q_{++}$ and the symmetry $\tilde\kappa(-\bar z)=-\overline{\tilde\kappa(z)}$ then implies that $\tilde\kappa(z)\in\overline\H$ for every $z\in\overline\H$. \\\hline
\end{tabular}
\end{center}
\caption{The Nachman--Smith--Waag model.}
\label{tbNacSmiWaa}
\end{table}


\section{The Integrated Photoacoustic Operator}
Let us now return to the attenuated photoacoustic imaging problem. Thus, we consider the operator mapping the source term $h$ in the attenuated wave equation \autoref{eqWaveEquationAttenuation} (interpreted in the sense of \autoref{eqSolAttWaveEq}) to the measurements, which shall correspond to the solution of the attenuated wave equation on the measurement surface $\partial\Omega$ measured for all time.

According to \autoref{thSolAttWave}, the solution $p$ of the attenuated wave equation is given by \autoref{eqSolAttWave}. This means that the temporal inverse Fourier transform of $p$ is the regular distribution corresponding to the function
\begin{equation}\label{eqSolAttWaveEqSimplified}
\check p(\omega,x) = \int_{\R^3}G_\kappa(\omega,x-y)h(y)\d y =-\frac{\i \omega}{4 \pi \sqrt{2 \pi}} \int_{\R^3}\frac{\e^{\i \kappa(\omega)|x-y|}}{|x-y|}h(y) \d y,\qquad\omega\in\R,\; x\in\R^3,
\end{equation}
where $G_\kappa$ is defined by \autoref{eqSolAttWaveKernel}.

We therefore introduce our measurements $\check m$ as the function
\[ \check m(\omega,\xi) = \check p(\omega,\xi)\quad\text{for all}\quad \omega\in\R,\;\xi\in\partial\Omega. \]

Instead of considering the operator mapping $h$ to $\check m$, we will divide the data by $-\i\omega$, meaning that we consider the map from $h$ to the inverse Fourier transform of the measurements which were integrated over time.
Additionally, we want to assume that the measurements are performed outside the support of the source.
\begin{remark}
This assumption that the absorption density functions $h$ has compact support in the domain $\Omega$ is very common in the theory of photoacoustics, see for instance \cite{AgrKucQui07,KucKun08}.
\end{remark}

\begin{definition}\label{deAttSphMean}
Let $\Omega\subset\R^3$ be a bounded Lipschitz domain and $\kappa$ be either a strong or a weak attenuation coefficient. For $\varepsilon>0$, we define $\Omega_\varepsilon=\{x\in\Omega\mid\dist(x,\partial\Omega)>\varepsilon\}$.

Then, we call
\begin{equation}\label{eqAttSphMean}
\PO:L^2(\Omega_\varepsilon)\to L^2(\R\times\partial\Omega),\quad \PO h(\omega,\xi) = \frac1{4\pi\sqrt{2\pi}}\int_{\Omega_\varepsilon}\frac{\e^{\i\kappa(\omega)|\xi-y|}}{|\xi-y|}h(y)\d y
\end{equation}
the \emph{integrated photoacoustic operator} of the attenuation coefficient $\kappa$ in frequency domain.
\end{definition}

\begin{lemma}\label{thPOBdd}
Let $\Omega\subset\R^3$ be a bounded Lipschitz domain and $\varepsilon>0$. Then, the integrated photoacoustic operator $\PO:L^2(\Omega_\varepsilon)\to L^2(\R\times\partial\Omega)$ of an attenuation coefficient $\kappa$ is a bounded linear operator and its adjoint is given by
\begin{equation}\label{eqPOadj}
\PO^*:L^2(\R\times\partial\Omega)\to L^2(\Omega_\varepsilon),\quad\PO^*\check m(y) = \frac1{4\pi\sqrt{2\pi}}\int_{-\infty}^\infty\int_{\partial\Omega} \frac{\e^{-\i\overline{\kappa(\omega)}|\xi-y|}}{|\xi-y|}\check m(\omega,\xi)\d S(\xi)\d\omega.
\end{equation}
\end{lemma}
\begin{proof}
\begin{enumerate}
\item[]
\item
We first consider the case of a strong attenuation coefficient. Then, we have for every $\omega\in\R$ and $\xi\in\partial\Omega$ the estimate
\[ |\PO h(\omega,\xi)|^2 \le \frac{|\Omega_\varepsilon|}{32\varepsilon^2\pi^3}\|h\|_2^2\e^{-2\varepsilon\Im\kappa(\omega)}. \]
Since, by \autoref{deStrongAttenuation}, $\Im\kappa(\omega)\ge \kappa_0|\omega|^\beta$ for all $\omega\in\R$ with $|\omega|\ge\omega_0$ for some sufficiently large~$\omega_0\ge0$, this shows that $\PO:L^2(\Omega_\varepsilon)\to L^2(\R\times\partial\Omega)$ is a bounded linear operator.
\item
In the case of a weak attenuation coefficient, see \autoref{deWeakAttenuation}, we split the operator into $\PO=\POa+\POb$, where we define
\begin{equation}\label{eqPOa}
\POa h(\omega,\xi) = \frac1{4\pi\sqrt{2\pi}}\int_{\Omega_\varepsilon}\frac{\e^{\i\frac\omega c|\xi-y|}}{|\xi-y|}\e^{-\kappa_\infty|\xi-y|}h(y)\d y
\end{equation}
as the photoacoustic operator with constant attenuation and the perturbation
\begin{equation}\label{eqPOb}
\POb h(\omega,\xi) = \frac1{4\pi\sqrt{2\pi}}\int_{\Omega_\varepsilon}\frac{\e^{\i\frac\omega c|\xi-y|}}{|\xi-y|}\e^{-\kappa_\infty|\xi-y|}(\e^{\i\kappa_*(\omega)|\xi-y|}-1)h(y)\d y.
\end{equation}

Now, $\POa h$ is seen to be the inverse Fourier transform of the function $\mathcal P_\kappa^{(0)}h$, defined by
\[ \mathcal P_\kappa^{(0)}h(t,\xi) = \frac{\e^{-\kappa_\infty ct}}{4\pi t}\int_{\Omega_\varepsilon\cap\partial B_{ct}(\xi)}h(z)\d S(z),\quad t>0,\;\xi\in\partial\Omega, \]
and $\mathcal P_\kappa^{(0)}h(t,\xi)=0$ for $t\le0$, $\xi\in\partial\Omega$.
Now, $\mathcal P_\kappa^{(0)}:L^2(\Omega_\varepsilon)\to L^2(\R\times\partial\Omega)$ can be directly seen to be a bounded linear operator, since we have, recalling that $\kappa_\infty\ge0$,
\begin{align*}
\|\mathcal P_\kappa^{(0)}h\|_2^2 &\le \int_{\partial\Omega}\int_0^\infty \frac{1}{16\pi^2t^2}
\left|\int_{\Omega_\varepsilon\cap\partial B_{ct(\xi)}} h(z)\d S(z) \right|^2 \d t \d S(\xi) \\
& \le \int_{\partial\Omega}\int_0^\infty \frac{c^2}{4\pi} \int_{\Omega_\varepsilon\cap\partial B_{ct(\xi)}} |h(z)|^2\d S(z) \d t \d S(\xi).
\end{align*}
Thus, combining the two inner integrals to an integral over $\Omega_\varepsilon$, we find that
\[ \|\mathcal P_\kappa^{(0)}h\|_2^2 \le \frac c{4\pi}|\partial\Omega|\|h\|_2^2. \]
This is a special case of the more general result in \cite[Lemma 4.1]{Pal10}.

Thus, $\mathcal P_\kappa^{(0)}:L^2(\Omega_\varepsilon)\to L^2(\R\times\partial\Omega)$ and therefore, because the Fourier transform on $L^2(\R)$ is an isometry, also $\POa:L^2(\Omega_\varepsilon)\to L^2(\R\times\partial\Omega)$ are bounded, linear operators.

For $\POb h$, we get the estimate
\begin{equation}\label{eqPobL2}
|\POb h(\omega,\xi)|^2 \le \frac{|\Omega_\varepsilon|}{32\varepsilon^2\pi^3}\|h\|_2^2\sup_{y\in\Omega_\varepsilon}|\e^{\i\kappa_*(\omega)|\xi-y|}-1|^2.
\end{equation}
We now remark that we can find for every bounded set $D\subset\C$ a constant $C>0$ such that
\begin{equation}\label{eqExponential}
|\e^z-1|\le C|z|\quad\text{for all}\quad z\in D.
\end{equation}
Therefore, since $\kappa_*$ is according to \autoref{deWeakAttenuation} bounded and $|\xi-y|$ remains bounded since $\Omega$ is bounded, we find a constant $\tilde C>0$ such that
\[ \sup_{y\in\Omega_\varepsilon}|\e^{\i\kappa_*(\omega)|\xi-y|}-1|^2 \le\tilde C|\kappa_*(\omega)|^2\quad\text{for all}\quad \omega\in\R,\;\xi\in\partial\Omega. \]
Since $\kappa_*$ is additionally square integrable by \autoref{deWeakAttenuation}, we find by inserting this into the estimate~\autoref{eqPobL2} that $\POb:L^2(\Omega_\varepsilon)\to L^2(\R\times\partial\Omega)$ is a bounded, linear operator, and therefore so is the operator $\PO:L^2(\Omega_\varepsilon)\to L^2(\R\times\partial\Omega)$.
\end{enumerate}
\end{proof}

To obtain the singular values of the operator $\PO$, we consider the operator $\PO^*\PO$ on $L^2(\Omega_\varepsilon)$. It turns out that this operator is a Hilbert--Schmidt integral operator, in particular therefore compact. So, by the singular theorem for compact operators, its spectrum consists of at most countably many positive eigenvalues and the value zero.
\begin{proposition}\label{thPO2compact}
Let $\PO:L^2(\Omega_\varepsilon)\to L^2(\R\times\partial\Omega)$ be the integrated photoacoustic operator of a weak or a strong attenuation coefficient $\kappa$ for some bounded, convex domain $\Omega\subset\R^3$ with smooth boundary and some $\varepsilon>0$. Then, $\PO^*\PO:L^2(\Omega_\varepsilon)\to L^2(\Omega_\varepsilon)$ is a self-adjoint integral operator with kernel $F\in L^2(\Omega_\varepsilon\times\Omega_\varepsilon)$ given by
\begin{equation}\label{eqPATOp2Kernel}
\POK(x,y)=\frac1{32\pi^3}\int_{-\infty}^\infty\int_{\partial\Omega}\frac{\e^{\i\kappa(\omega)|\xi-y|-\i\overline{\kappa(\omega)}|\xi-x|}}{|\xi-y||\xi-x|}\d S(\xi)\d\omega,
\end{equation}
that is
\begin{equation}\label{eqPATOp2}
\PO^*\PO h(x)=\int_{\Omega_\varepsilon}\POK(x,y)h(y)\d y.
\end{equation}

In particular, $\PO^*\PO$ is a Hilbert--Schmidt operator and thus compact.
\end{proposition}
We remark that the convexity and smoothness assumptions on $\Omega$ are only needed for the weak attenuation case. For strong attenuation, a Lipschitz domain $\Omega$ is sufficient.
\begin{proof}
The representation \autoref{eqPATOp2Kernel} of the integral kernel $\POK$ of the operator $\PO^*\PO$ is directly obtained by combining the formulas \autoref{eqAttSphMean} and \autoref{eqPOadj} for $\PO$ and $\PO^*$. To prove that $F\in L^2(\Omega_\varepsilon\times\Omega_\varepsilon)$, we treat the two cases of strong and weak attenuation coefficients separately.
\begin{enumerate}
\item
For a strong attenuation coefficient, we estimate directly
\[ |\POK(x,y)|^2 \le \frac{|\partial\Omega|}{32\varepsilon^2\pi^3}\int_{-\infty}^\infty\e^{-2\varepsilon\Im\kappa(\omega)}\d\omega\quad\text{for all}\quad x,y\in\Omega_\varepsilon. \]
According to \autoref{deStrongAttenuation}, we have $\Im\kappa(\omega)\ge\kappa_0|\omega|^\beta$ for all $|\omega|\ge\omega_0$ for some $\omega_0\ge0$ and therefore, $|F_\kappa|^2$ is uniformly bounded. Thus, $F_\kappa\in L^2(\Omega_\varepsilon\times\Omega_\varepsilon)$, which implies that $\PO^*\PO$ is a Hilbert--Schmidt operator and compact, see for example \cite[Theorems 6.10 and 6.11]{Wei80}.
\item
In the case of a weak attenuation coefficient, we write $\POK$ similar to the proof of \autoref{thPOBdd} as the sum of a contribution $\POKa$ of a medium with constant attenuation and perturbations $\POKb$ and $\POKc$: $\POK=\POKa+\POKb+\POKc$ with
\begin{equation}\label{eqPOKsplit}
F_\kappa^{(j)}(x,y) = \frac1{32\pi^3}\int_{-\infty}^\infty\int_{\partial\Omega}\frac{\e^{\i\frac\omega c(|\xi-y|-|\xi-x|)}}{|\xi-y||\xi-x|}\e^{-\kappa_\infty(|\xi-y|+|\xi-x|)}f^{(j)}_\kappa(\omega,\xi,x,y)\d S(\xi)\d\omega,
\end{equation}
where
\begin{align*}
\POka(\omega,\xi,x,y) &= 1, \\
\POkb(\omega,\xi,x,y) &= \i\kappa_*(\omega)|\xi-y|-\i\overline{\kappa_*(\omega)}|\xi-x|, \\
\POkc(\omega,\xi,x,y) &= \e^{\i\kappa_*(\omega)|\xi-y|-\i\overline{\kappa_*(\omega)}|\xi-x|}-\POkb(\omega,\xi,x,y)-\POka(\omega,\xi,x,y).
\end{align*}

\begin{itemize}
\item
To calculate the double integral in $\POKa$, we first parametrise $\partial\Omega$ depending on the values $x,y\in\Omega_\varepsilon$ with $x\ne y$.
		
We choose a positively oriented, orthonormal basis $(e_j)_{j=1}^3\subset\R^3$ with $e_3=\frac{y-x}{|y-x|}$ and consider the curve
\[ \Gamma_\varphi=\partial\Omega\cap E_\varphi,\quad E_\varphi=\{\xi\in\R^3\mid\left<\xi-\tfrac12(x+y),\cos\varphi\,e_2-\sin\varphi\,e_1\right>=0\}, \]
given as the intersection of the boundary $\partial\Omega$ and the plane $E_\varphi$ through $\frac12(x+y)$, spanned by the vectors $e_3$ and $\cos\varphi\,e_1+\sin\varphi\,e_2$. 

\begin{figure}[tb]
\begin{center}
\begin{tikzpicture}[x=0.6cm, y=0.6cm,font=\footnotesize]
\draw(-2,-1) circle (1pt) node[left,yshift=0.5ex] {$x$};
\draw(2,1) circle (1pt) node[right,yshift=-0.5ex] {$y$};
\draw(0,0) circle (1pt) node[right,yshift=-1ex] {$\tfrac12(x+y)$};

\draw[->] (0,0)--(1,0.5) node[above] {$e_3$}; 
\draw[->] (0,0)--(-0.5,1) node[left] {$\cos\varphi\,e_1+\sin\varphi\,e_2$}; 

\draw[name path=boundary](0,0) circle (9 and 3);
\draw(7,-2.2) node[right] {$\Gamma_\varphi=\Gamma_{\varphi-\pi}$};

\path[name path=orthogonalLine](3,-6) -- (-3,6);
\path[name intersections={of=boundary and orthogonalLine}];
\draw[densely dotted](intersection-1) -- (intersection-2);

\path[name path=radialCoordinate]($(0,0)-(6,3)$) -- ($(-3,6)-(6,3)$);
\path[name intersections={of=boundary and radialCoordinate}];
\coordinate(psi) at (intersection-1);
\coordinate(psiPi) at (intersection-2);
\draw[thick](psi) -- (-6,-3);
\draw[decorate,decoration={brace,amplitude=3pt,raise=1pt}](psi) -- (-6,-3) node[pos=0.5,anchor=west,yshift=1ex,xshift=0.5ex] {$r(\varphi,z)$};
\draw[decorate,decoration={brace,amplitude=3pt,mirror,raise=1pt}](psiPi) -- (-6,-3) node[pos=0.5,anchor=east,yshift=-0.5ex,xshift=-0.5ex] {$-r(\varphi-\pi,z)$};
\draw[decorate,decoration={brace,amplitude=3pt,raise=1pt}] (0,0) -- (-6,-3) node[pos=0.5,below,xshift=0.5ex] {$-z$};
\draw[thick](-6,-3) -- (0,0);
\draw(psi) circle(1pt) node[above,xshift=-1.5ex] {$\psi(\varphi,z)$};
\draw(psiPi) circle(1pt) node[right,xshift=-0.5ex,yshift=1.5ex] {$\psi(\varphi-\pi,z)$};
\draw(-6,-3) circle(1pt) node[right,xshift=-1ex,yshift=-2ex] {$\tfrac12(x+y)+z\,e_3$};

\coordinate(tangentPoint) at (8.877545,0.49319696);
\draw(tangentPoint) circle (1pt);
\draw($(0,0)-(tangentPoint)$) circle (1pt);
\path[name path=XY](-8,-4) -- (8,4);
\path[name path=tangent]($(tangentPoint)-(-2,4)$) -- ($(tangentPoint)+(-2,4)$);
\path[name intersections={of=XY and tangent}];
\coordinate(XYtangent) at (intersection-1);

\draw[densely dashed]($(0,0)-(XYtangent)-(0.2,0.1)$) -- ($(XYtangent)+(0.2,0.1)$);
\draw[densely dotted]($(XYtangent)+(-0.1,0.2)$)--($(tangentPoint)+(0.2,-0.4)$);
\draw[densely dotted]($(0.1,-0.2)-(XYtangent)$)--($(-0.2,0.4)-(tangentPoint)$);
\draw(XYtangent) circle(1pt) node[right] {$\tfrac12(x+y)+b_\varphi e_3$};
\draw($(0,0)-(XYtangent)$) circle(1pt) node[left] {$\tfrac12(x+y)+a_\varphi e_3$};

\pgfresetboundingbox
\clip (-11.5,-4.5) rectangle (12.5,4.5);

\end{tikzpicture}
\end{center}
\caption{\label{fig:parametrisation} Parametrisation of the intersection $\Gamma_\varphi$, $\varphi\in(0,\pi)$, of the boundary $\partial\Omega$ and the plane $E_\varphi$.}
\end{figure}

Setting
\[ a_\varphi=\min_{\xi\in\Gamma_\varphi}\left<\xi-\tfrac12(x+y),e_3\right>\quad\text{and}\quad b_\varphi=\max_{\xi\in\Gamma_\varphi}\left<\xi-\tfrac12(x+y),e_3\right>, \]
we choose the parametrisation $\psi\in C^1(U;\R^3)$ of $\partial\Omega$ (up to a set of measure zero) defined on the open set $U=\{(\varphi,z)\mid z\in(a_\varphi,b_\varphi),\varphi\in(-\pi,0)\cup(0,\pi)\}$ as
\begin{equation}\label{eqPOKaParametrisation}
\psi(\varphi,z) = \tfrac12(x+y)+r(\varphi,z)(\cos\varphi\,e_1+\sin\varphi\,e_2)+z\,e_3,
\end{equation}
where we pick for $\varphi\in(0,\pi)$ the function $r$ in such a way that $r(\varphi,z)>-r(\varphi-\pi,z)$ so that the two maps $\psi(\varphi-\pi,\cdot)$ and $\psi(\varphi,\cdot)$ parametrise together $\Gamma_\varphi$, see \autoref{fig:parametrisation}.

\item
If we would formally interchange the order of integration in the definition \autoref{eqPOKsplit} of $\POKa$, the integral over $\omega$ would lead to a $\delta$-distribution at the zeros of the exponent $\frac1c(|\xi-x|-|\xi-y|)$. We therefore start by analysing this exponent, which is up to the prefactor $\frac1c$ given by
\[ g(\varphi,z) = |\psi(\varphi,z)-x|-|\psi(\varphi,z)-y| \]
if we use the parametrisation $\psi$ for integrating over $\partial\Omega$.

The zeros of $g$ are exactly those points $(\varphi,z)$ such that $\psi(\varphi,z)$ is in the bisection plane of $x$ and $y$. Thus, we have by construction of the parametrisation $\psi$, see \autoref{eqPOKaParametrisation}, that 
\begin{equation}\label{eqPOKaZeroPhase}
g(\varphi,z)=0\quad\text{is equivalent to}\quad z=0.
\end{equation}

Furthermore, we can prove that
\begin{equation}\label{eqPOKaDerivativePhase}
\partial_zg(\varphi,z) = \left<\frac{\psi(\varphi,z)-x}{|\psi(\varphi,z)-x|}-\frac{\psi(\varphi,z)-y}{|\psi(\varphi,z)-y|},\partial_z\psi(\varphi,z)\right>
\end{equation}
only vanishes at the two points where $\psi(\varphi,z)$ is the intersection point of the line through $x$ and~$y$ with $\partial\Omega$: We assume by contradiction that $\partial_zg(\varphi,z)=0$ at a point $(\varphi,z)\in U$ with $\psi(\varphi,z)$ not lying on the line through $x$ and $y$. Then, the first vector $\frac{\psi(\varphi,z)-x}{|\psi(\varphi,z)-x|}-\frac{\psi(\varphi,z)-y}{|\psi(\varphi,z)-y|}$
would be a non-zero vector in $E_\varphi$, and the second vector $\partial_z\psi(\varphi,z)$
is by construction a non-zero tangent vector on $\Gamma_\varphi\subset E_\varphi$ at $\psi(\varphi,z)$. Thus, $\partial_zg(\varphi,z)=0$ would imply that the first vector is a non-trivial multiple of the outer unit normal vector $\nu(\psi(\varphi,z))$ to $\Gamma_\varphi$ at $\psi(\varphi,z)$. However, if $w_1,w_2\in\R^2$ are two unit vectors with $w_1-w_2=n$, $n\ne0$, then $\left<w_1,n\right>=-\left<w_2,n\right>$, since for given $w_1\in S^1$, $w_2$ is the intersection point of $S^1$ with the line parallel to $n$ through $w_1$. Thus, we would have
\[ \left<\frac{\psi(\varphi,z)-x}{|\psi(\varphi,z)-x|},\nu(\psi(\varphi,z))\right> = -\left<\frac{\psi(\varphi,z)-y}{|\psi(\varphi,z)-y|},\nu(\psi(\varphi,z))\right>, \]
but, because of the convexity of $\Omega$, the projections $\left<\frac{\psi-x}{|\psi-x|},\nu\circ\psi\right>$ and $\left<\frac{\psi-y}{|\psi-y|},\nu\circ\psi\right>$ have to be both positive, which is a contradiction. Therefore, $\partial_zg(\varphi,z)=0$ if and only if $\psi(\varphi,z)$ is on the line through $x$ and $y$.

\item
After these preparations, we can now reduce the formula \autoref{eqPOKsplit} for $\POKa$ to a one-dimensional integral and estimate it explicitly to show that $\POKa\in L^2(\Omega_\varepsilon\times\Omega_\varepsilon)$.

We plug in the parametrisation $\psi$ into the definition \autoref{eqPOKsplit} of $\POKa$ and find for $x\ne y$
\begin{equation}\label{eqPOKaParametrisedIntegral}
\POKa(x,y) = \int_{-\pi}^\pi\int_{-\infty}^\infty\int_{a_\varphi}^{b_\varphi}\e^{\i\omega g(\varphi,z)}\mu(\varphi,z,x,y)\d z\d\omega\d\varphi,
\end{equation}
where 
\begin{equation}\label{eqPOKaMu}
\mu(\varphi,z,x,y)=\frac1{32\pi^3}\frac{\e^{-\kappa_\infty(|\psi(\varphi,z)-y|+|\psi(\varphi,z)-x|)}}{|\psi(\varphi,z)-y||\psi(\varphi,z)-x|}\sqrt{\det(\d\psi^{\mathrm T}(\varphi,z)\d\psi(\varphi,z))}.
\end{equation}

To evaluate the integrals, we remark that if $\lambda\in C^1(\R)$ is a real-valued, strictly monotone function with $\lambda(\R)=I\subset\R$ and $\rho\in L^1(\R^2)$, then
\[ \int_{-\infty}^\infty\int_{-\infty}^\infty \e^{\i\omega\lambda(z)}\rho(z)\d z\d\omega = \int_{-\infty}^\infty\int_{-\infty}^\infty\e^{\i\omega\zeta}\rho_\lambda(\zeta)\d\zeta\d\omega,\quad\text{where}\quad\rho_\lambda(\zeta)=\frac{\rho(\lambda^{-1}(\zeta))}{|\lambda'(\lambda^{-1}(\zeta))|}\chi_I(\zeta) \]
with the characteristic function $\chi_I$ of the interval $I$. Now, the inner integral is up to the missing factor $\frac1{\sqrt{2\pi}}$ exactly the inverse Fourier transform $\check\rho_\lambda$ of $\rho_\lambda$ so that we get
\[ \int_{-\infty}^\infty\int_{-\infty}^\infty \e^{\i\omega\lambda(z)}\rho(z)\d z\d\omega = \sqrt{2\pi}\int_{-\infty}^\infty\check\rho_\lambda(\omega)\d\omega = 2\pi\rho_\lambda(0). \]

Applying this result to the two inner integrals in \autoref{eqPOKaParametrisedIntegral}, where we use from above that $g(\varphi,\cdot)$ has only two critical points and is therefore piecewise strictly monotone to first split the innermost integral into integrals over intervals where $\partial_zg(\varphi,z)\ne0$, we find with \autoref{eqPOKaZeroPhase} that
\begin{equation}\label{eqPOKaSimplified}
\POKa(x,y) = 2\pi\int_{-\pi}^\pi\frac{\mu(\varphi,0,x,y)}{|\partial_zg(\varphi,0)|}\d\varphi.
\end{equation}
Evaluating \autoref{eqPOKaDerivativePhase} at $z=0$, we find with $|\psi(\varphi,0)-x|=|\psi(\varphi,0)-y|$, see \autoref{eqPOKaZeroPhase}, and the explicit formula \autoref{eqPOKaParametrisation} for the parametrisation $\psi$ that
\[ \partial_zg(\varphi,0)= \left<\frac{y-x}{|\psi(\varphi,0)-x|},\partial_zr(\varphi,0)(\cos\varphi\,e_1+\sin\varphi\,e_2)+e_3\right>=\frac{|y-x|}{|\psi(\varphi,0)-x|}. \]
Plugging this together with formula \autoref{eqPOKaMu} for $\mu$ into \autoref{eqPOKaSimplified}, we finally get for $\POKa$ the representation
\begin{equation}\label{eqPOKaIntegralIntersection}
\POKa(x,y) = \frac1{16\pi^2|y-x|}\int_{-\pi}^\pi\frac{\e^{-2\kappa_\infty|\psi(\varphi,0)-x|}}{|\psi(\varphi,0)-x|}\sqrt{\det(\d\psi^{\mathrm T}(\varphi,0)\d\psi(\varphi,0))}\d\varphi.
\end{equation}

This is an integral over the intersection $\gamma_{x,y}=\{\xi\in\partial\Omega\mid \left<\xi-\frac12(x+y),y-x\right>=0\}$ of $\partial\Omega$ and the bisection plane of the points $x$ and $y$. To express it in a parametrisation invariant form, we write it as a line integral over $\gamma_{x,y}$, which means we want to use the volume element $|\partial_\varphi\psi(\varphi,0)|\d\varphi$. Calculating the outer unit normal vector explicitly from \autoref{eqPOKaParametrisation}, we find
\[ \nu\circ\psi = \frac{\partial_\varphi\psi\times\partial_z\psi}{\sqrt{\det(\d\psi^{\mathrm T}\d\psi)}} = \frac{r\,e_1+\partial_\varphi r\,e_2-r\partial_zr\,e_3}{\sqrt{\det(\d\psi^{\mathrm T}\d\psi)}}. \]
Thus, we get the relation
\[ |\partial_\varphi\psi|^2 = r^2+(\partial_\varphi r)^2 = \det(\d\psi^{\mathrm T}\d\psi)\left|\nu\circ\psi-\left<\nu\circ\psi,e_3\right>e_3\right|^2. \]
With this, we can write \autoref{eqPOKaIntegralIntersection} as the parametrisation free line integral
\begin{equation}\label{eqPOKaExplicit}
\POKa(x,y) = \frac1{16\pi^2|y-x|}\int_{\gamma_{x,y}}\frac{\e^{-2\kappa_\infty|\xi-x|}}{|\xi-x|}\left|\nu(\xi)-\left<\frac{y-x}{|y-x|},\nu(\xi)\right>\frac{y-x}{|y-x|}\right|^{-1}\d S(\xi).
\end{equation}

In particular, we have
\begin{equation}\label{eqPOKaBound}
|\POKa(x,y)| \le \frac A{|x-y|}\quad\text{for all}\quad x,y\in\Omega_\varepsilon\quad\text{with}\quad x\ne y
\end{equation}
for some constant $A>0$, and therefore $\POKa\in L^2(\Omega_\varepsilon\times\Omega_\varepsilon)$.

\item
To estimate the first perturbation $\POKb$, we remark that, according to \autoref{deWeakAttenuation}, $\kappa_*$ is bounded and square integrable. We can therefore pull in the definition \autoref{eqPOKsplit} of $\POKb$ the integration over the variable $\omega$ as an inverse Fourier transform inside the surface integral and find that
\begin{multline*}
\POKb(x,y) = \frac1{16\pi^2\sqrt{2\pi}}\int_{\partial\Omega}\i\e^{-\kappa_\infty(|\xi-y|+|\xi-x|)} \\
\times\left(\frac{\check\kappa_*(\tfrac1c(|\xi-y|-|\xi-x|))}{|\xi-x|}-\frac{\overline{\check\kappa_*(\tfrac1c(|\xi-x|-|\xi-y|))}}{|\xi-y|}\right)\d S(\xi).
\end{multline*}
Choosing now a radius $R>\diam\Omega$, we get by applying Hölder's inequality and increasing the domain of integration that
\[ \int_{\Omega_\varepsilon}\int_{\Omega_\varepsilon}|\POKb(x,y)|^2\d x\d y \le \frac{|\partial\Omega|}{128\pi^5\varepsilon^2}\int_{\partial\Omega}\int_{B_R(\xi)}\int_{B_R(\xi)}|\check\kappa_*(\tfrac1c(|\xi-y|-|\xi-x|))|^2\d x\d y\d S(\xi) \]
Thus, switching in the two inner integrals to spherical coordinates around the point $\xi$, we find
\[ \int_{\Omega_\varepsilon}\int_{\Omega_\varepsilon}|\POKb(x,y)|^2\d x\d y \le \frac{|\partial\Omega|^2R^4}{8\pi^3\varepsilon^2}\int_0^R\int_0^R|\check\kappa_*(\tfrac1c(r-\rho))|^2\d r\d\rho \le \frac{|\partial\Omega|^2R^5c}{8\pi^3\varepsilon^2}\|\check\kappa_*\|_2^2, \]
which shows that $\POKb\in L^2(\Omega_\varepsilon\times\Omega_\varepsilon)$.

\item
The second perturbation $\POKc$ can be bounded directly via
\begin{equation}\label{eqPOKcEstimate}
|\POKc(x,y)| \le \frac1{32\pi^3\varepsilon^2}\int_{-\infty}^\infty\int_{\partial\Omega}|\POkc(\omega,\xi,x,y)|\d S(\xi)\d\omega
\end{equation}
for all $x,y\in\Omega_\varepsilon$. Using that for every bounded set $D\subset\C$, there exists a constant $C>0$ such that
\[ |\e^z-z-1|\le C|z|^2\quad\text{for all}\quad z\in D \]
holds, we find a constant $\tilde C>0$ so that
\begin{equation}\label{eqPOkcEstimate}
|\POkc(\omega,\xi,x,y)|\le \tilde C|\kappa_*(\omega)|^2\quad\text{for all}\quad \omega\in\R,\;\xi\in\partial\Omega,\;x,y\in\Omega_\varepsilon.
\end{equation}
Plugging this into \autoref{eqPOKcEstimate}, we get that
\[ |\POKc(x,y)| \le \frac{\tilde C|\partial\Omega|}{32\pi^3\varepsilon^2}\|\kappa_*\|_2^2\quad\text{for all}\quad x,y\in\Omega_\varepsilon. \]
Thus, in particular, we have $\POKc\in L^2(\Omega_\varepsilon\times\Omega_\varepsilon)$.
\end{itemize}

We conclude therefore that $\POK=\POKa+\POKb+\POKc\in L^2(\Omega_\varepsilon\times\Omega_\varepsilon)$, which shows as in the first part of the proof that $\PO^*\PO$ is a Hilbert--Schmidt operator and compact.
\end{enumerate}
\end{proof}

\section{Singular Values of the Integrated Photoacoustic Operator}
We have seen in \autoref{thPO2compact} that the operator $\PO^*\PO$, given by \autoref{eqPATOp2}, is a compact operator. The inversion of the photoacoustic problem is therefore ill-posed. To quantify the ill-posedness, we want to study the decay of the eigenvalues $(\lambda_n(\PO^*\PO))_{n\in\N}$ of $\PO^*\PO$, where we enumerate the eigenvalues in decreasing order: $0\le\lambda_{n+1}(\PO^*\PO)\le\lambda_n(\PO^*\PO)$ for all $n\in\N$.

We differ again between the two cases of a strong and of a weak attenuation coefficient $\kappa$.

\subsection{Strongly Attenuating Media}

To obtain the behaviour of the eigenvalues of $\PO^*\PO$ in the case of a strong attenuation coefficient~$\kappa$, see \autoref{deStrongAttenuation}, we will use \autoref{thExponentialDecay} which gives a criterion for a general integral operator with smooth kernel to have exponentially fast decaying eigenvalues in terms of an upper bound on the derivatives of the kernel, see \autoref{eqKernelUpperBound}. We therefore only have to check that the kernel \autoref{eqPATOp2Kernel} of $\PO^*\PO$ fulfils these estimates. The calculations are straightforward (although a bit tedious) and can be found explicitly in \autoref{seDerivativeGreensFunction}.

\begin{proposition}\label{thForwardOperatorKernelUpperBound}
Let $\Omega$ be a bounded Lipschitz domain in $\R^3$, $\varepsilon>0$ and $\PO:L^2(\Omega_\varepsilon)\to L^2(\R\times\partial\Omega)$ be the integrated photoacoustic operator of a strong attenuation coefficient $\kappa$.

Then, the kernel $F_\kappa$ of $\PO^*\PO$, explicitly given by \autoref{eqPATOp2Kernel}, fulfils the estimate
\begin{equation}\label{eqForwardOperatorKernelUpperBound}
\frac1{j!}\sup_{x,y\in\Omega_\varepsilon}\sup_{v\in S^2}\left|\Dj F_\kappa(x,y+sv)\right|\le Bb^jj^{(\frac N\beta-1)j}\quad\text{for all}\quad j\in\N_0,
\end{equation}
for some constants $B,b>0$, where $N\in\N$ denotes the exponent for $\ell=0$ in the condition \autoref{eqAttCoeffPolBdd} and $\beta\in(0,N]$ is the exponent in the condition \autoref{eqStrongAttenuation} for the strong attenuation coefficient $\kappa$.
\end{proposition}
\begin{proof}
Putting the derivatives with respect to $s$ inside the integrals in the definition \autoref{eqPATOp2Kernel} of the kernel~$F_\kappa$, we get that
\[ \Dj F_\kappa(x,y+s v) = \int_{-\infty}^\infty\frac1{\omega^2}\int_{\partial\Omega}\overline{G_\kappa(\omega,x-\xi)}\Dj G_\kappa(\omega,y-\xi+s v)\d S(\xi)\d\omega, \]
where $G_\kappa$ denotes the integral kernel \autoref{eqSolAttWaveKernel}. We remark that the term $\frac1{\omega^2}$ comes from the fact that we consider the integrated photoacoustic operator instead of the operator which maps directly the measurements to the initial data.

Using \autoref{thGreensFunctionDirectionalDerivative} to estimate the derivative of $G_\kappa$, we find a constant $C>0$ so that
\begin{multline*}
\frac1{j!}\left|\Dj F_\kappa(x,y+sv)\right| \\
\le C^j\int_{-\infty}^\infty\frac1{\omega^2}\int_{\partial\Omega}|G_\kappa(\omega,x-\xi)|\,|G_\kappa(\omega,y-\xi)|\left(\frac1{|y-\xi|^j}+\frac1{j!}|\kappa(\omega)|^j\right)\d S(\xi)\d\omega.
\end{multline*}
From the uniform estimate
\[ |G_\kappa(\omega,x-\xi)|\le\frac{|\omega|\e^{-\varepsilon\Im\kappa(\omega)}}{4\pi\varepsilon\sqrt{2\pi}}\quad\text{for all}\quad x\in\Omega_\varepsilon,\;\xi\in\partial\Omega,\;\omega\in\R, \]
which is directly obtained from the definition \autoref{eqSolAttWaveKernel} of $G_\kappa$ by using that $|x-\xi|\ge\varepsilon$ for all $\xi\in\partial\Omega$ and $x\in\Omega_\varepsilon$, it then follows that
\[ \frac1{j!}\left|\Dj F_\kappa(x,y+sv)\right| \le \frac{|\partial\Omega|C^j}{32\pi^3\varepsilon^2}\int_{-\infty}^\infty\e^{-2\varepsilon\Im\kappa(\omega)}\left(\frac1{\varepsilon^j}+\frac1{j!}|\kappa(\omega)|^j\right)\d\omega. \]

Applying now \autoref{thFrequencyIntegral} (for the first term in the integrand, we use \autoref{thFrequencyIntegral} with $j=0$), we find constants $B,b>0$ so that
\[ \frac1{j!}\left|\Dj F(x,y+sv)\right| \le Bb^jj^{(\frac N\beta-1)j}\quad\text{for all}\quad x,y\in\Omega_\varepsilon,\;v\in S^2,\;j\in\N_0. \]
\end{proof}

Combining \autoref{thForwardOperatorKernelUpperBound} with \autoref{thExponentialDecay}, we obtain the decay of the singular values of the integrated photoacoustic operator.
\begin{corollary}
Let $\Omega$ be a bounded Lipschitz domain in $\R^3$, $\varepsilon>0$, and $\PO:L^2(\Omega_\varepsilon)\to L^2(\R\times\partial\Omega)$ be the integrated photoacoustic operator of a strong attenuation coefficient $\kappa$.

Then, there exist constants $C,c>0$ so that the eigenvalues $(\lambda_n(\PO^*\PO))_{n\in\N}$ of $\PO^*\PO$ in decreasing order fulfil
\begin{equation}\label{eqForwardOperatorDecayRate}
\lambda_n(\PO^*\PO)\le Cn\sqrt[m]n\exp\left(-cn^{\frac\beta{Nm}}\right)\quad\text{for all}\quad n\in\N.
\end{equation}
\end{corollary}
\begin{proof}
According to \autoref{thForwardOperatorKernelUpperBound}, we know that there exist constants $B,b>0$ so that the integral kernel $F_\kappa$ of the operator $\PO^*\PO$ fulfils the estimate \autoref{eqForwardOperatorKernelUpperBound}. Applying thus \autoref{thExponentialDecay} with $\mu=\frac N\beta-1$ to the operator $\PO^*\PO$, we obtain the decay rate \autoref{eqForwardOperatorDecayRate}.
\end{proof}

\subsection{Weakly Attenuating Media}
To analyse the operator $\PO^*\PO$ in the case of a weak attenuation coefficient $\kappa$, see \autoref{deWeakAttenuation}, we split $\PO$ as in the proof of \autoref{thPOBdd} in $\PO=\POa+\POb$, see \autoref{eqPOa} and \autoref{eqPOb}. We will show that decomposing the operator as $\PO^*\PO=\POa^*\POa+\mathcal Q_\kappa$, the eigenvalues of the operator $\mathcal Q_\kappa=\POa^*\POb+\POb^*\POa+\POb^*\POb$ decay faster than those of $\POa^*\POa$ so that $\mathcal Q_\kappa$ does not alter the asymptotic decay rate of the eigenvalues of $\POa^*\POa$.

The term $\POa^*\POa$ corresponds to a constant attenuation and its behaviour was already discussed in \cite{Pal10}.
\begin{lemma}\label{thEvNonAtt}
Let $\kappa$ be a weak attenuation coefficient, $\Omega\subset\R^3$ be a bounded, convex domain with smooth boundary, and $\varepsilon>0$. We define the operator $\POa:L^2(\Omega_\varepsilon)\to L^2(\Omega_\varepsilon)$ by \autoref{eqPOa}. Then, there exist constants $C_1,C_2>0$ such that we have
\begin{equation}\label{eqEvNonAtt}
C_1n^{-\frac23}\le\lambda_n(\POa^*\POa)\le C_2n^{-\frac23}\quad\text{for all}\quad n\in\N.
\end{equation}
\end{lemma}
\begin{proof}
The idea of the proof is to show that the operator $\POa^*\POa$ has the same eigenvalues as an elliptic pseudofifferential operator $\mathcal T:L^2(M)\to L^2(M)$ of order $-2$ on a closed manifold $M$. Then, we can apply the result \cite[Theorem 15.2]{Shu87} to obtain the asymptotic behaviour of the eigenvalues.

\begin{itemize}
\item
First, we want to replace the operator $\POa^*\POa$ by a pseudodifferential operator $\mathcal T$ on a closed manifold with the same eigenvalues.
		
We have seen in the proof of \autoref{thPO2compact} that the operator $\POa^*\POa$ is an integral operator with integral kernel $\POKa$ defined by \autoref{eqPOKsplit}. We now generate the closed  manifold $M$ by taking two copies of $\overline{\Omega_\varepsilon}$ and identifying their boundary points: $M=(\overline{\Omega_\varepsilon}\times\{1,2\})/\sim$ with the equivalence relation $(x,a)\sim(\tilde x,\tilde a)$ if and only if $x=\tilde x$ and either $a=\tilde a$ or $x\in\partial\Omega_\varepsilon$. This is called the double of the manifold with boundary $\overline{\Omega_\varepsilon}$, see for example \cite[Example 9.32]{Lee13}. Then, the operator $\mathcal T:L^2(M)\to L^2(M)$ given by
\[ \mathcal Th([x,a]) = \frac12\sum_{b=1}^2\int_{\Omega_\varepsilon}\POKa(x,y)h([y,b])\d y \]
has the same non-zero eigenvalues as $\POa^*\POa:L^2(\Omega_\varepsilon)\to L^2(\Omega_\varepsilon)$. Indeed, if $h$ is an eigenfunction of $\mathcal T$ with eigenvalue $\lambda\ne0$, then necessarily $h([x,1])=h([x,2])$ for almost every $x\in\Omega_\varepsilon$ and therefore $x\mapsto h([x,1])$ is an eigenfunction of $\POa^*\POa$ with eigenvalue $\lambda$. Conversely, if $h$ is an eigenfunction of $\POa^*\POa$ with eigenvalue $\lambda$, then clearly $[x,a]\mapsto h(x)$ is an eigenfunction of $\mathcal T$ with eigenvalue $\lambda$.

To write $\mathcal T$ in the form of a pseudodifferential operator, we extend the kernel $\POKa$ to a smooth function $\tilde F_\kappa^{(0)}\in C^\infty(\Omega_\varepsilon\times\R^3)$ by choosing an arbitrary cut-off function $\phi\in C^\infty_{\mathrm c}(\R^3)$ with $\phi(y)=1$ for $y\in\Omega_\varepsilon$ and $\supp\phi\subset\Omega$ and setting
\begin{equation}\label{eqEvNonAttTildeF}
\tilde F_\kappa^{(0)}(x,y) = \frac{\phi(y)}{32\pi^3}\int_{-\infty}^\infty\int_{\partial\Omega}\frac{\e^{-\kappa_\infty(|\xi-y|+|\xi-x|)}}{|\xi-y||\xi-x|}\e^{\i\frac\omega c(|\xi-y|-|\xi-x|)}\d S(\xi)\d\omega,\quad x\in\Omega_\varepsilon,\;y\in\R^3.
\end{equation}
Then, defining $g$ up to the normalisation factor $(2\pi)^{\frac32}$ as the inverse Fourier transform of $\tilde F_\kappa^{(0)}$ with respect to $y$:
\begin{equation}\label{eqEvNonAttDefG}
g(x,k) = \int_{\R^3}\tilde F_\kappa^{(0)}(x,y)\e^{-\i\left<k,x-y\right>}\d y,\quad x\in\Omega_\varepsilon,\;k\in\R^3,
\end{equation}
we can write the kernel $\POKa$ with the Fourier inversion theorem in the form
\[ \POKa(x,y) = \frac1{(2\pi)^3}\int_{\R^3}g(x,k)\e^{\i\left<k,x-y\right>}\d k,\quad x,y\in\Omega_\varepsilon, \]
where $g$ is smooth, since $g(x,\cdot)$ is the Fourier transform of a function with compact support.
		
\item
In the expression \autoref{eqEvNonAttDefG} for $g$, the integral over $\omega$ from the definition \autoref{eqEvNonAttTildeF} of $\tilde F_\kappa^{(0)}$ can be seen as a one-dimensional inverse Fourier transform, which allows us to get rid of two one-dimensional integrals.

To this end, we pull the outer integral over $\R^3$ inside both other integrals and write it in spherical coordinates around the point $\xi$: $y=\xi+r\theta$ with $r>0$ and $\theta\in S^2$. This gives us
\begin{multline*}
g(x,k) = \frac1{32\pi^3}\int_{\partial\Omega}\e^{\i\left<k,\xi-x\right>}\int_{S^2}\int_{-\infty}^\infty\e^{-\i\frac\omega c|\xi-x|} \\
\times\int_0^\infty\frac{r\e^{-\kappa_\infty(r+|\xi-x|)}}{|\xi-x|}\phi(\xi+r\theta)\e^{\i(\frac\omega c+\left<k,\theta\right>)r}\d r\d\omega\d S(\theta)\d S(\xi).
\end{multline*}
For every $\xi\in\partial\Omega$ and every $\theta\in S^2$, the two inner integrals with respect to $r$ and $\omega$ each represent a Fourier transform and we get with $\rho(r)=\frac{r\e^{-\kappa_\infty(r+|\xi-x|)}}{|\xi-x|}\phi(\xi+r\theta)\chi_{[0,\infty)}(r)$ that
\begin{align*}
\int_{-\infty}^\infty\e^{-\i\frac\omega c|\xi-x|}\int_0^\infty\rho(r)\e^{\i(\frac\omega c+\left<k,\theta\right>)r}\d r\d\omega &= \sqrt{2\pi}\int_{-\infty}^\infty\check\rho(\tfrac\omega c+\left<k,\theta\right>)\e^{-\i\frac\omega c|\xi-x|}\d\omega \\
&=2\pi c\e^{\i|\xi-x|\left<k,\theta\right>}\rho(|\xi-x|).
\end{align*}
Thus, we find
\begin{equation}\label{eqEvNonAttG}
g(x,k) = \frac c{16\pi^2}\int_{\partial\Omega}\int_{S^2}\e^{-2\kappa_\infty|\xi-x|}\phi(\xi+|\xi-x|\theta)\e^{\i(|\xi-x|\left<k,\theta\right>+\left<k,\xi-x\right>)}\d S(\theta)\d S(\xi).
\end{equation}

\item
We are now interested in the leading order asymptotics of $g(x,k)$ as $|k|\to\infty$. To obtain this, we will apply the stationary phase method, see for example \cite[Theorem 7.7.5]{Hoe03}.
		
So, let $\psi\in C^\infty(U;\R^3)$ be a parametrisation of $\partial\Omega$ and $\Theta\in C^\infty(V;\R^3)$ be a paramtetrisation of $S^2$ with some open sets $U,V\subset\R^2$. Then, according to the stationary phase method, the asymptotics is determined by the region around the critical points of the phase function
\[ \Phi_{x,k}(\eta,\vartheta) = |\psi(\eta)-x|\left<k,\Theta(\vartheta)\right>+\left<k,\psi(\eta)-x\right> \]
in the integrand in~\autoref{eqEvNonAttG}. The optimality conditions
\[ 0 = \partial_{\vartheta_i}\Phi_{x,k}(\eta,\vartheta) = |\psi(\eta)-x|\left<k,\partial_{\vartheta_i}\Theta(\vartheta)\right>,\quad i=1,2, \]
with respect to $\vartheta$ imply that $k$ is normal to the tangent space of $S^2$ in the point $\Theta(\vartheta)$ at a critical point $(\eta,\vartheta)$ of $\Phi_{x,k}$, that is $\Theta(\vartheta)=\pm\frac{k}{|k|}$. The optimality conditions
\[ 0 = \partial_{\eta_i}\Phi_{x,k}(\eta,\vartheta) = \left<\partial_{\eta_i}\psi(\eta),\pm|k|\frac{\psi(\eta)-x}{|\psi(\eta)-x|}+k\right>,\quad i=1,2, \]
with respect to $\eta$ then imply for a critical point $(\eta,\vartheta)$ of $\Phi_{x,k}$ that the projections of the two vectors $-\frac{\psi(\eta)-x}{|\psi(\eta)-x|}$ and $\Theta(\vartheta)=\pm\frac k{|k|}$ on the tangent space of $\partial\Omega$ at $\psi(\eta)$ coincide. Since both vectors have unit length, this means that up to the sign of the normal component they have to be equal. Additionally, we use that, because of the cut-off term $\phi(\psi(\eta)+|\psi(\eta)-x|\Theta(\vartheta))$ in the integrand, the critical points at which the vector $\Theta(\vartheta)$ points outwards the domain $\Omega$ at $\psi(\eta)$ do not contribute to the integral. Therefore, a relevant critical point $(\eta,\theta)$ is such that $\Theta(\vartheta)$ is pointing inwards at $\psi(\eta)$ and since $-\frac{\psi(\eta)-x}{|\psi(\eta)-x|}$ is also pointing inwards, we are left with the two critical points $(\eta^{(\ell)},\vartheta^{(\ell)})$ given by
\[ \Theta(\vartheta^{(\ell)}) = (-1)^\ell\frac k{|k|} = -\frac{\psi(\eta^{(\ell)})-x}{|\psi(\eta^{(\ell)})-x|},\quad\ell=1,2. \]
In particular, we have $\Phi_{x,k}(\eta^{(\ell)},\vartheta^{(\ell)})=0$.

For the second derivatives of $\Phi_{x,k}$ at the critical points, we find
\begin{align*}
\partial_{\eta_i\eta_j}\Phi_{x,k}(\eta^{(\ell)},\vartheta^{(\ell)}) &= \frac{(-1)^\ell|k|}{|\psi(\eta)-x|}\left(\left<\partial_{\eta_i}\psi(\eta^{(\ell)}),\partial_{\eta_i}\psi(\eta^{(\ell)})\right>-\left<\partial_{\eta_i}\psi(\eta^{(\ell)}),\frac k{|k|}\right>\left<\partial_{\eta_j}\psi(\eta^{(\ell)}),\frac k{|k|}\right>\right), \\
\partial_{\vartheta_i\vartheta_j}\Phi_{x,k}(\eta^{(\ell)},\vartheta^{(\ell)}) &= |\psi(\eta^{(\ell)})-x|\left<k,\partial_{\vartheta_i\vartheta_j}\Theta(\vartheta^{(\ell)})\right>, \\
\partial_{\eta_i\vartheta_j}\Phi_{x,k}(\eta^{(\ell)},\vartheta^{(\ell)}) &= 0.
\end{align*}

For the determinants of the derivatives with respect to $\eta$ and $\vartheta$, we obtain (this can be readily checked for parametrisations $\psi$ and $\Theta$ corresponding to normal coordinates at the points $(\psi(\eta^{(\ell)}),\Theta(\vartheta^{(\ell)}))$)
\begin{align*}
\det(\partial_{\eta_i\eta_j}\Phi_{x,k}(\eta^{(\ell)},\vartheta^{(\ell)}))_{i,j=1}^2 &= \frac{\left<k,\nu(\psi(\eta^{(\ell)}))\right>^2}{|\psi(\eta^{(\ell)})-x|^2}\det(\d\psi^{\mathrm T}(\eta^{(\ell)})\d\psi(\eta^{(\ell)})), \\
\det(\partial_{\vartheta_i\vartheta_j}\Phi_{x,k}(\eta^{(\ell)},\vartheta^{(\ell)}))_{i,j=1}^2 &= |k|^2\det(\d\Theta^{\mathrm T}(\vartheta^{(\ell)})\d\Theta(\vartheta^{(\ell)})),
\end{align*}
where $\nu:\partial\Omega\to S^2$ denotes the outer unit normal vector field on $\partial\Omega$.

Therefore, the stationary phase method, see for example \cite[Theorem 7.7.5]{Hoe03}, implies for $x\in\Omega_\varepsilon$, $k\in\R^3$, and $\mu>0$ that we have asymptotically for $\mu\to\infty$
\begin{align*}
g(x,\mu k)&=\frac c{16\pi^2}\int_{U\times V}\e^{-2\kappa_\infty|\psi(\eta)-x|}\phi(\psi(\eta)+|\psi(\eta)-x|\Theta(\vartheta))\e^{\i\mu\Phi_{x,k}(\eta,\vartheta)} \\
&\qquad\qquad\times\sqrt{\det(\d\psi^{\mathrm T}(\eta)\d\psi(\eta))\det(\d\Theta^{\mathrm T}(\vartheta)\d\Theta(\vartheta))}\d(\eta,\vartheta) \\
&= \frac c{4|k|\mu^2}\sum_{\ell=1}^2\frac{|\psi(\eta^{(\ell)})-x|}{|\left<k,\nu(\psi(\eta^{(\ell)}))\right>|}\e^{-2\kappa_\infty|\psi(\eta^{(\ell)})-x|}+\Ord\left(\frac1{\mu^3}\right).
\end{align*}
Thus, $g$ is of the form
\[ g(x,k)=g_{-2}(x,k)+\Ord(|k|^{-3}) \]
with $g_{-2}(x,\cdot)$ being a positive function which is homogeneous of order $-2$.
\end{itemize}

Therefore, a parameterix of the pseudodifferential operator $\mathcal T$ on $L^2(M)$ is an elliptic pseudodifferential operator of order $2$ and thus has, according to \cite[Theorem 15.2]{Shu87}, eigenvalues which grow as $n^{\frac23}$. Consequently, the eigenvalues of $\mathcal T$ and thus also those of $\PO^*\PO$ decay as $n^{-\frac23}$.
\end{proof}

To estimate the eigenvalues of the term $\POb^*\POb$ in $\PO^*\PO$, we show with Mercer's theorem that the operator is trace class.

\begin{lemma}\label{thEvPer}
Let $\kappa$ be a weak attenuation coefficient, $\Omega\subset\R^3$ be a bounded, convex domain with smooth boundary, and $\varepsilon>0$. We define the operator $\POb:L^2(\Omega_\varepsilon)\to L^2(\Omega_\varepsilon)$ by \autoref{eqPOb}. Then, we have that
\begin{equation}\label{eqEvPer}
\lim_{n\to\infty}n\lambda_n(\POb^*\POb)=0.
\end{equation}
\end{lemma}
\begin{proof}
Using the definition \autoref{eqPOb} of $\POb$, we find that
\[ \POb^*\POb h(x) = \int_{\Omega_\varepsilon}R_\kappa(x,y)h(y)\d y \]
with the integral kernel
\begin{align*}
R_\kappa(x,y) &= \frac1{32\pi^3}\int_{-\infty}^\infty\int_{\partial\Omega}r_\kappa(\xi,\omega,x,y)\d S(\xi)\d\omega, \\
r_\kappa(\xi,\omega,x,y) &= \frac{\e^{\i\frac\omega c(|\xi-y|-|\xi-x|)}}{|\xi-y||\xi-x|}\e^{-\kappa_\infty(|\xi-y|+|\xi-x|)}(\e^{\i\kappa_*(\omega)|\xi-y|}-1)(\e^{-\i\overline{\kappa_*(\omega)}|\xi-x|}-1).
\end{align*}
Since for every bounded set $D\subset\C$, there exists a constant $C$ such that
\[ |\e^z-1| \le C|z|\quad\text{for all}\quad z\in D, \]
we find a constant $\tilde C>0$ such that the integrand $r_\kappa$ is uniformly estimated by an integrable function:
\[ |r_\kappa(\xi,\omega,x,y)| \le \tilde C|\kappa_*(\omega)|^2\quad\text{for all}\quad x,y\in\overline{\Omega_\varepsilon},\;\xi\in\partial\Omega,\;\omega\in\R. \]
Taking now an arbitrary sequence $(x_k,y_k)_{k\in\N}\subset\Omega_\varepsilon$ converging to an element $(x,y)\in\overline{\Omega_\varepsilon}$, we get with the dominated convergence theorem and the continuity of $r_\kappa$ that
\[ \lim_{k\to\infty}R_\kappa(x_k,y_k) = \frac1{32\pi^3}\int_{-\infty}^\infty\int_{\partial\Omega}\lim_{k\to\infty}r_\kappa(\xi,\omega,x_k,y_k)\d S(\xi)\d\omega = R_\kappa(x,y). \]
Thus, $R_\kappa$ is continuous and therefore Mercer's theorem, see for example \cite[Chapter~III, \S5 and \S9]{CouHil53}, implies that
\[ \sum_{n=1}^\infty\lambda_n(\POb^*\POb)<\infty. \]
Since $(\lambda_n(\POb^*\POb))_{n=1}^\infty$ is by definition a decreasing sequence, Abel's theorem, see for example \cite[§173]{Har21}, gives us \autoref{eqEvPer}.
\end{proof}

From the decay rates of the singular values of $\POa$ and $\POb$, we can directly deduce the decay rate of the perturbation $\PO^*\PO-\POa^*\POa$.
\begin{lemma}\label{thEvDif}
Let $\kappa$ be a weak attenuation coefficient, $\Omega\subset\R^3$ be a bounded, convex domain with smooth boundary, and $\varepsilon>0$.

Then, the operator
\begin{equation}\label{eqEvDifQ}
\mathcal Q_\kappa:L^2(\Omega_\varepsilon)\to L^2(\Omega_\varepsilon),\quad \mathcal Q_\kappa = \PO^*\PO-\POa^*\POa
\end{equation}
with $\POa:L^2(\Omega_\varepsilon)\to L^2(\Omega_\varepsilon)$ being defined by \autoref{eqPOa} fulfils
\begin{equation}\label{eqEvDif}
\lim_{n\to\infty}n^{\frac56}|\lambda_n(\mathcal Q_\kappa)| = 0.
\end{equation}
Here $(\lambda_n(\mathcal Q_\kappa))_{n\in\N}$ denotes the eigenvalues of $\mathcal Q_\kappa$, sorted in decreasing order: $|\lambda_{n+1}(\mathcal Q_\kappa)|\le|\lambda_n(\mathcal Q_\kappa)|$ for all $n\in\N$.
\end{lemma}

\begin{proof}
We start with the positive semi-definite operator $\PO^*\PO$. We split $\PO$ as in the proof of \autoref{thPOBdd} in $\PO=\POa+\POb$ with $\POb$ given by \autoref{eqPOb}. Then, we can write $\mathcal Q_\kappa$ in the form
\[ \mathcal Q_\kappa=\PO^*\PO-\POa^*\POa = \POb^*\POa+\POa^*\POb+\POb^*\POb. \]

To estimate the eigenvalues of the operator $\POb^*\POa+\POa^*\POb$, we use that for all $m,n\in\N$ the inequalities
\begin{align*}
|\lambda_{m+n-1}(\POb^*\POa+\POa^*\POb)| &\le s_m(\POa^*\POb)+s_n(\POa^*\POb)\quad\text{and} \\
s_{m+n-1}(\POa^*\POb) &\le s_m(\POa)s_n(\POb)
\end{align*}
hold, see for example \cite[Chapter II.2.3, Corollary 2.2]{GohKre69}, where $s_n(T)=\sqrt{\lambda_n(T^*T)}$ denotes the singular values of a compact operator~$T$ sorted in decreasing order: $s_{n+1}(T)\le s_n(T)$ for all $n\in\N$. Here, we used that $s_n(T)=s_n(T^*)$, see for example \cite[Chapter II.2.2]{GohKre69}. Inserting the decay rates \autoref{eqEvNonAtt} and \autoref{eqEvPer} for $(s_n(\POa))_{n=1}^\infty$ and $(s_n(\POb))_{n=1}^\infty$ into these inequalities, we find that
\[ \lim_{n\to\infty}n^{\frac56}|\lambda_n(\POb^*\POa+\POa^*\POb)| = 0. \]
Estimating, again with \cite[Chapter II.2.3, Corollary 2.2]{GohKre69}, the eigenvalues of the sum $\mathcal Q_\kappa$ of the two operators $\POb^*\POa+\POa^*\POb$ and $\POb^*\POb$, we find that
\[ |\lambda_{m+n-1}(\mathcal Q_\kappa)| \le |\lambda_m(\POb^*\POa+\POa^*\POb)|+\lambda_n(\POb^*\POb) \]
for all $m,n\in\N$, which yields with the behaviour \autoref{eqEvPer} of the eigenvalues of $\POb^*\POb$ the result \autoref{eqEvDif}.
\end{proof}

Combining \autoref{thEvNonAtt} and \autoref{thEvDif}, we obtain that the singular values of the photoacoustic operator $\PO$ decay as in the unperturbed case.
\begin{theorem}
Let $\PO:L^2(\Omega_\varepsilon)\to L^2(\R\times\partial\Omega)$ be the integrated photoacoustic operator of a weak attenuation coefficient $\kappa$ for some bounded, convex domain $\Omega\subset\R^3$ with smooth boundary and some $\varepsilon>0$. Then, there exist constants $C_1,C_2>0$ such that we have
\begin{equation}\label{eqEvPO}
C_1n^{-\frac23}\le\lambda_n(\PO^*\PO)\le C_2n^{-\frac23}\quad\text{for all}\quad n\in\N.
\end{equation}
\end{theorem}
\begin{proof}
Defining again the operators $\POa$, see \autoref{eqPOa}, and $\mathcal Q_\kappa$, see \autoref{eqEvDifQ}, we know from \cite[Chapter II.2.3, Corollary 2.2]{GohKre69} for all $m,n\in\N$ that
\[ \lambda_{m+n-1}(\PO^*\PO) \le \lambda_m(\POa^*\POa^*)+|\lambda_n(\mathcal Q_\kappa)| \]
and
\[ \lambda_{m+n-1}(\POa^*\POa)-|\lambda_n(\mathcal Q_\kappa)| \le \lambda_m(\PO^*\PO^*). \]
Therefore, \autoref{thEvNonAtt} and \autoref{thEvDif} imply bounds of the form \autoref{eqEvPO}.
\end{proof}


\appendix

\section{Eigenvalues of Integral Operators of Hilbert--Schmidt Type}
In this section, we derive estimates for the eigenvalues of operators $T$ of the form
\begin{equation}\label{eqIntegralOp}
T:L^2(U) \to L^2(U),\quad (Th)(x) = \int_{U} F(x,y)h(y)\d y
\end{equation}
on a bounded, open set $U\subset\R^m$ with an Hermitian integral kernel $F\in C(\bar U\times\bar U)$ (that is, $F(x,y)=\overline{F(y,x)}$). In particular, such an operator $T$ is a self-adjoint Hilbert--Schmidt operator, and we want to additionally assume that $T$ is positive semi-definite. So, the eigenvalues $(\lambda_n(T))_{n\in\N}$ of $T$ are non-negative and we enumerate them in decreasing order:
\[ 0\le\lambda_{n+1}(T)\le\lambda_n(T)\quad\text{for all}\quad n\in\N. \]

To obtain the asymptotic decay rate of the eigenvalues of such an operator $T$, we proceed as in~\cite{ChaHa99} where a characterisation for a decay rate of the form $\lambda_n(T)=\Ord(n^{-k})$ was presented in terms of an upper estimate on the derivatives of the kernel $F$. The extension to an exponential decay rate is rather straightforward.

First, we show that when approximating an operator $T_1$ by a finite rank operator $T_2$, we can estimate the eigenvalues above the rank of the finite rank operator $T_2$ in terms of the supremum norm of the difference of their kernels, see for example \cite[Satz II]{Wey12}.

\begin{lemma}\label{thFiniteRankPerturbation}
Let $U\subset\R^m$ be a bounded, open set, $T_i:L^2(U)\to L^2(U)$, $i=1,2$, be two integral operators with Hermitian integral kernels $F_i\in C(\bar U\times\bar U)$. Moreover, let $T_1$ be positive semi-definite and $T_2$ have finite rank $r\in\N_0$.

Then, the eigenvalues $(\lambda_n(T_1))_{n\in\N}$ of $T_1$ (sorted in decreasing order) satisfy
\begin{equation*}
\sum_{n=r+1}^\infty\lambda_n(T_1)\le(2r+1)|U|\|F_1-F_2\|_\infty.
\end{equation*}
\end{lemma}
\begin{proof}
The min-max theorem (see for example \cite[Chapter II.2.3]{GohKre69}) states that for every self-adjoint,
compact operator $T$ and every fixed $m\in\N$
\begin{equation}
\label{eq:GohKre}
|\lambda_m(T)| = \min_{\Rank(A) \le m-1}\|T-A\|,
\end{equation}
where the minimum is taken over all operators $A:L^2(U)\to L^2(U)$ with rank less than or equal to $m-1$.

Let us fix $n\in\N$ now. Applying \autoref{eq:GohKre} with $T=T_1-T_2$ shows that there exists an
operator $A$ with $\Rank(A) \leq n-1$ such that
\begin{equation}
\label{eq:T1T2A}
|\lambda_n(T_1-T_2)|=\|T_1-T_2-A\|.
\end{equation}
Because $\Rank(T_2) \leq r$ and $\Rank(A) \leq n-1$,
$\Rank(T_2+A) \leq n+r-1$, and we therefore have
\begin{equation}\label{eqT1T2A1}
\|T_1-T_2-A\| \ge \min_{\Rank(\tilde A) \leq n+r-1}\|T_1-\tilde A\|.
\end{equation}
Using \autoref{eq:GohKre} with $T=T_1$ and $m=n+r$ we find that
\begin{equation}\label{eqT1T2A2}
\min_{\Rank(\tilde A) \leq n+r-1}\|T_1-\tilde A\|=|\lambda_{n+r}(T_1)|=\lambda_{n+r}(T_1),
\end{equation}
since $T_1$ is positive semi-definite. Combining the three relations~\autoref{eq:T1T2A}, \autoref{eqT1T2A1}, and \autoref{eqT1T2A2}, we get
\[ \lambda_{n+r}(T_1) \le |\lambda_n(T_1-T_2)|\quad\text{for all}\quad r,n\in\N. \]
Taking the sum over all $n \in \N$, we get
\begin{equation}
\label{eq:nt1t1}
\sum_{n=r+1}^\infty\lambda_n(T_1) \le \sum_{n=1}^\infty|\lambda_n(T_1-T_2)|.
\end{equation}

The eigenvalues of $T_1-T_2$ do not need to be all non-negative, however, since $T_2$ has rank at most~$r$, the operator $T_1-T_2$ cannot have more than $r$ negative eigenvalues. Moreover, their norm is bounded by
\begin{equation}
\label{eq:ast}
|\lambda_n(T_1-T_2)|\le \|T_1-T_2\|\le \|F_1-F_2\|_{L^2(U\times U)}\le|U|\|F_1-F_2\|_\infty.
\end{equation}
Thus, we can estimate the sum in \autoref{eq:nt1t1} by
\begin{equation}
\label{eq:nt1t1b}
\begin{split}
\sum_{n=1}^\infty|\lambda_n(T_1-T_2)| &= \sum_{n=1}^\infty\lambda_n(T_1-T_2)+2\sum_{\lambda_n(T_1-T_2)<0}|\lambda_n(T_1-T_2)| \\
&\le \sum_{n=1}^\infty\lambda_n(T_1-T_2)+2r|U|\|F_1-F_2\|_\infty.
\end{split}
\end{equation}

Moreover, choosing an orthonormal eigenbasis $(\psi_n)_{n=1}^\infty\subset L^2(U)$ of the compact, self-adjoint operator $T_1-T_2$, we get with Mercer's theorem, see for example \cite[Chapter~III, §5 and §9]{CouHil53}, that
\[ F_1(x,y)-F_2(x,y) = \sum_{n=1}^\infty\lambda_n(T_1-T_2)\psi_n(x)\overline{\psi_n(y)}, \]
and therefore
\begin{equation}\label{eq:nt1t1c}
\sum_{n=1}^\infty\lambda_n(T_1-T_2) = \int_U(F_1(x,x)-F_2(x,x))\d x \le |U|\|F_1-F_2\|_\infty.
\end{equation}

Combining \autoref{eq:nt1t1}, \autoref{eq:nt1t1b}, and \autoref{eq:nt1t1c} gives
\begin{equation*}
\sum_{n=r+1}^\infty\lambda_n(T_1) \le (2r+1)|U|\|F_1-F_2\|_\infty.
\end{equation*}
\end{proof}

Thus, approximating the kernel in our integral operator by one of its Taylor polynomials, we get a convergence rate for the eigenvalues depending on the approximation error of the Taylor polynomial. To improve this estimate, we first subdivide the domain $U$ in smaller domains, so that the approximation error of the Taylor polynomial is smaller.

Regarding an upper bound for the eigenvalues, it is indeed enough to keep the subdomains along the diagonal of $U\times U$, see \cite[Lemma 1]{ChaHa99}.

\begin{lemma}\label{thDiagonalBlocks}
Let $U\subset\R^m$ be a bounded, open set and $T_1:L^2(U)\to L^2(U)$ be a positive semi-definite integral operator with Hermitian kernel $F_1\in C(\bar U\times\bar U)$. Let further $Q_\ell\subset U$, $\ell=1,\ldots,N$, be open, pairwise disjoint sets such that $U\subset\bigcup_{\ell=1}^N\bar Q_\ell$ and define the kernel $F_2:\bar U\times\bar U\to\C$ by
\begin{equation*}
F_2 = F_1\sum_{\ell=1}^N\chi_{Q_\ell\times Q_\ell}.
\end{equation*}

Then, the integral operator $T_2:L^2(U)\to L^2(U)$ with the integral kernel $F_2$ is also positive semi-definite and fulfils
\begin{equation}\label{eqDiagonalBlocksSmallEv}
\sum_{n=r+1}^\infty\lambda_n(T_1)\le\sum_{n=r+1}^\infty\lambda_n(T_2)\quad\text{for every}\quad r\in\N_0.
\end{equation}
\end{lemma}

\begin{proof}
For each $\ell\in\{1,\ldots,N\}$, let $P_\ell:L^2(U)\to L^2(U)$, $P_\ell h=h\chi_{Q_\ell}$ be the orthogonal projection onto the subspace $L^2(Q_\ell)$.
In particular, because the sets $Q_\ell$ are pairwise disjoint, we have
\begin{equation}\label{eq:proj}
P_kP_\ell=\delta_{k,\ell}P_\ell.
\end{equation}
With this notation, we can write
\begin{equation}\label{eqDiagonalBlocksT2Proj}
T_2=\sum_{\ell=1}^NP_\ell T_1P_\ell.
\end{equation}

Now, we first show that $T_2$ is indeed positive semi-definite. Let us assume by contradiction that this is not the case. Then, there exists a function $h\in L^2(U)$ so that $\left<h,T_2h\right><0$. Thus, because of the representation~\autoref{eqDiagonalBlocksT2Proj} of $T_2$, we find an index $\ell\in\{1,\ldots,N\}$ such that
\[ \left<P_\ell h,T_1P_\ell h\right> = \left<h,P_\ell T_1P_\ell h\right> < 0. \]
However, this contradicts the fact that $T_1$ should be positive semi-definite.

To get a relation between the eigenvalues of $T_1$ and $T_2$, we construct a sequence $(T^{(k)})_{k=0}^N$ of positive semi-definite operators interpolating between $T^{(0)}=T_1$ and $T^{(N)}=T_2$. We define recursively for every $k\in\{1,\ldots,N\}$
\begin{equation}
\label{eq:T2}
T^{(k)}=\frac12[T^{(k-1)}+(1-2P_k)T^{(k-1)}(1-2P_k)]\quad\text{with}\quad T^{(0)}=T_1.
\end{equation}
Before continuing, we want to verify that this definition indeed yields $T^{(N)}=T_2$. We first remark that, because of the orthogonality relation \autoref{eq:proj}, the equation \autoref{eq:T2} can be written as
\begin{alignat*}{2}
T^{(k)}P_k&=P_kT^{(k-1)}P_k\quad&&\text{and} \\
T^{(k)}P_\ell &= (1-P_k)T^{(k-1)}P_\ell\quad&&\text{for}\quad\ell\ne k.
\end{alignat*}
Thus, we get recursively for every $\ell$ that
\begin{align*}
T^{(N)}P_\ell &= (1-P_N)T^{(N-1)}P_\ell\\
&=(1-P_N)\cdots(1-P_{\ell+1})T^{(\ell)}P_\ell\\
&=(1-P_N)\cdots(1-P_{\ell+1})P_{\ell}T^{(\ell-1)}P_\ell\\
&=(1-P_N)\cdots(1-P_{\ell+1})P_\ell(1-P_{\ell-1})\cdots(1-P_1)T_1P_\ell\\
&=P_\ell T_1P_\ell.
\end{align*}
So, again using the representation \autoref{eqDiagonalBlocksT2Proj} of $T_2$, we see that
\[ T^{(N)} = T^{(N)}\sum_{\ell=1}^NP_\ell = \sum_{\ell=1}^NP_\ell T_1P_\ell = T_2. \]

Now, by Ky Fan's maximum principle, see for example \cite[Chapter II.4]{GohKre69}, we can write the sum of the $r\in\N$ largest eigenvalues of $T^{(k)}$, $k\in\{1,\ldots,N\}$, in the form
\[ \sum_{n=1}^r\lambda_n(T^{(k)}) = \sup\left\{\sum_{n=1}^r\left<h_n,T_2^{(k)}h_n\right>\mid (h_n)_{n=1}^r\subset L^2(U),\;\left<h_n,h_{n'}\right>=\delta_{n,n'}\right\}, \]
Inserting the recursive definition \autoref{eq:T2} for $T^{(k)}$, we get from the subadditivity of the supremum the estimate
\[ \sum_{n=1}^r\lambda_n(T^{(k)})\le\frac12\sum_{n=1}^r\left[\lambda_n\big(T^{(k-1)}\big)+\lambda_n\Big((1-2P_k)T^{(k-1)}(1-2P_k)\Big)\right]. \]

Since $(1-2P_k)^2=1$ and eigenvalues are invariant under conjugation (that is, we have $\lambda_n(AT^{(k-1)}A^{-1})=\lambda_n(T^{(k-1)})$ for every invertible operator $A$), this simplifies to
\[ \sum_{n=1}^r\lambda_n(T^{(k)}) \le \sum_{n=1}^r\lambda_n(T^{(k-1}). \]
We therefore get recursively the inequality
\begin{equation}\label{eqDiagonalBlocksLargeEv}
\sum_{n=1}^r\lambda_n(T_2)\le\sum_{n=1}^r\lambda_n(T_1).
\end{equation}

Additionally, since $h$ is an eigenfunction of $T_2$ if and only if the functions $P_\ell h$ are for every $\ell\in\{1,\ldots,N\}$ either zero or an eigenfunction of $P_\ell T_1P_\ell$ with the same eigenvalue, we have that
\[ \sum_{n=1}^\infty\lambda_n(T_2)=\sum_{\ell=1}^N\sum_{n'=1}^\infty\lambda_{n'}(P_\ell T_1P_\ell). \]
According to Mercer's theorem, see for example \cite[Chapter~III, §5 and §9]{CouHil53}, we therefore get as in \autoref{eq:nt1t1c} that
\begin{equation}\label{eqDiagonalBlocksTrace}
\sum_{n=1}^\infty\lambda_n(T_2) = \sum_{\ell=1}^N\int_{Q_\ell}F_1(x,x)\d x=\int_UF_1(x,x)\d x = \sum_{n=1}^\infty\lambda_n(T_1).
\end{equation}
Finally, combining \autoref{eqDiagonalBlocksLargeEv} and \autoref{eqDiagonalBlocksTrace}, we obtain the estimate \autoref{eqDiagonalBlocksSmallEv}.
\end{proof}

Now, putting together \autoref{thFiniteRankPerturbation} and \autoref{thDiagonalBlocks}, we obtain a decay rate for the eigenvalues of the integral operator depending on the convergence rate of the Taylor series of its kernel.

\begin{proposition}\label{thDecayRatesGeneral}
Let $U\subset\R^m$ be a bounded, open set, $T:L^2(U)\to L^2(U)$ be a positive semi-definite integral operator with an Hermitian kernel $F\in C^k(\bar U\times\bar U)$, $k\in\N$, and define
\begin{equation*}
M_j=\frac1{j!}\sup_{x,y\in U}\sup_{v\in S^{m-1}}\left|\Dj F(x,y+sv)\right|,\quad j\in\N,\;j\le k.
\end{equation*}

Then, there exist constants $A>0$ and $a>0$ such that for every $n\in\N$
\begin{equation}\label{eqEigenvalueGeneralUpperBound}
\lambda_n(T) \le A\min_{j\in J_{k,n}}\left[M_j\left(a\sqrt[m]{\frac2n}\right)^j(j+m)^{j+m}\right],
\end{equation}
where we take the minimum over all values $j$ in the set
\[ J_{k,n} =\left\{j\in\N\mid a(j+m)\le\sqrt[m]{\frac n2},\;j\le k\right\}. \]
\end{proposition}
\begin{proof}
For some $\delta\in(0,1)$, we partition the domain $U$ in pairwise disjoint open sets $Q_\ell$, $\ell=1,\ldots,N$, with diameter not greater than $\delta$ such that $\bigcup_{\ell=1}^N\bar Q_\ell\supset U$. We remark that there exists a constant $a>0$ so that we can find for every $\delta\in(0,1)$ such a partition with $N$ sets where
\begin{equation}\label{eqRelDeltaN}
N<\left(\frac a\delta\right)^m
\end{equation}
(for example by picking a cube with side length $D=\diam(U)$ containing $U\subset\R^m$ and choosing a partition in $\lceil\frac DL\rceil^m$ cubes of side length $L=\frac\delta{\sqrt m}$, which gives an estimate of the form \autoref{eqRelDeltaN} with $a=D\sqrt m+1$).

According to \autoref{thDiagonalBlocks}, we can now get an upper bound for the behaviour of the lower eigenvalues of $T$ by considering the eigenvalues of the integral operator $\tilde T:L^2(U)\to L^2(U)$ with kernel $F\sum_{\ell=1}^N\chi_{Q_\ell\times Q_\ell}$ or, equivalently, the eigenvalues of the integral operators $T_\ell:L^2(Q_\ell)\to L^2(Q_\ell)$ with the integral kernels $F\chi_{Q_\ell\times Q_\ell}$.

To obtain an estimate for the eigenvalues of the operators $T_\ell$, we consider instead of $T_\ell$ the finite rank operator which we get by approximating the kernel $F$ on $Q_\ell\times Q_\ell$ by a polynomial and then apply \autoref{thFiniteRankPerturbation}, see \cite[§2]{Wey12}.

So, we pick in every set $Q_\ell$ an arbitrary point $z_\ell$ and expand $F$ on $Q_\ell\times Q_\ell$ in a Taylor polynomial of degree $j-1$ for some $j\le k$ with respect to the second variable around the points $z_\ell$. Then, we get
\[ F(x,y) = F_{j,\ell}(x,y)+C_{j,\ell}(x,y),\quad x,y\in Q_\ell, \]
with the Taylor polynomial $F_{j,\ell}$ explicitly given by
\[ F_{j,\ell}(x,y) = \sum_{\{\alpha\in\N_0^m\mid|\alpha|\le j-1\}}\frac1{\alpha!}\partial_y^\alpha F(x,z_\ell)(y-z_\ell)^\alpha, \]
and with the remainder term $C_{j,\ell}$, which can be uniformly estimated by
\begin{equation}\label{eqDecayRatesGeneralRemainder}
|C_{j,\ell}(x,y)|\le M_j\delta^j\quad\text{for all}\quad x,y\in Q_\ell.
\end{equation}
Since $F_{j,\ell}$ is not necessarily Hermitian, we symmetrise it by defining the kernel $\tilde F_{j,\ell}$ on $Q_\ell\times Q_\ell$ as
\[ \tilde F_{j,\ell}(x,y) = \frac12(F_{j,\ell}(x,y)+\overline{F_{j,\ell}(y,x)}). \]

Then, $\tilde F_{j,\ell}$ is of the form $\tilde F_{j,\ell}(x,y)=\sum_{i=1}^{r_j}a_{i,\ell}(x)b_{i,\ell}(y)$ for some functions $a_{i,\ell},b_{i,\ell}\in C(\bar Q_\ell)$ with~$r_j$ given by two times the number of elements in the set $\{\alpha\in\N_0^m\mid|\alpha|\le j-1\}$, which is $r_j=2\binom{j+m-1}m$. Thus, the integral operator $\tilde T_{j,\ell}:L^2(Q_\ell)\to L^2(Q_\ell)$ with kernel $\tilde F_{j,\ell}$ has a finite rank which is not grater than $r_j$. Moreover, we get from \autoref{eqDecayRatesGeneralRemainder} the uniform estimate
\[ \sup_{x,y\in Q_\ell}|F(x,y)-\tilde F_{j,\ell}(x,y)| \le M_j\delta^j. \]

Therefore, \autoref{thFiniteRankPerturbation} gives us directly that
\begin{equation}\label{eqEvSingleBlock}
\sum_{n=r_j+1}^\infty\lambda_n(T_\ell) \le (2r_j+1)M_j|Q_\ell|\delta^j.
\end{equation}
Now, since every eigenvalue of the integral operator $\tilde T$ corresponds exactly to one eigenvalue of one of the operators $T_\ell$, we have that
\begin{equation}\label{eqEvBlockMatrix}
\sum_{n=1}^\infty\lambda_n(\tilde T) = \sum_{\ell=1}^N\sum_{n=1}^\infty\lambda_n(T_\ell),
\end{equation}
and since the eigenvalues of every operator are enumerated in decreasing order, we have that
\begin{equation}\label{eqEvBlockMatrixEstimate}
\sum_{n=1}^{Nr_j}\lambda_n(\tilde T) \ge \sum_{\ell=1}^N\sum_{n=1}^{r_j}\lambda_n(T_\ell).
\end{equation}

Thus, we get from \autoref{thDiagonalBlocks} for the eigenvalues of the operator $T$ by combining \autoref{eqEvBlockMatrix}, \autoref{eqEvBlockMatrixEstimate}, and using the estimate \autoref{eqEvSingleBlock} that
\begin{equation}\label{eqEvEstimateGeneral}
\sum_{n=Nr_j+1}^\infty\lambda_n(T) \le \sum_{n=Nr_j+1}^\infty\lambda_n(\tilde T) \le \sum_{\ell=1}^N\sum_{n=r_j+1}^\infty\lambda_n(T_\ell) \le (2r_j+1)M_j|U|\delta^j.
\end{equation}

For fixed $n\in\N$ and $j\in\N$, we now want to choose the parameter $N\in\N$ in such a way that $Nr_j<n$ and that we can make the parameter $\delta$ as small as possible. We pick
\[ \delta = a\sqrt[m]{\frac{r_j}n}, \]
where we assume that $j$ is chosen such that $r_j<\frac n{a^m}$, so that $\delta<1$ is fulfilled (an upper bound on $\delta$ is needed for an estimate of the form \autoref{eqRelDeltaN}). Because of $r_j=2\binom{j+m-1}{m}\le2(j+m-1)^m$, this condition on~$r_j$ can be ensured by imposing
\begin{equation}\label{eqEvBoundOnRj}
j+m\le\frac1a\sqrt[m]{\frac n2}.
\end{equation}

Then, according to \autoref{eqRelDeltaN}, there exists a partition $(Q_\ell)_{\ell=1}^N$ with
\[ N < \left(\frac a\delta\right)^m = \frac n{r_j}. \]

Evaluating \autoref{eqEvEstimateGeneral} at these parameters, we find that
\[ \lambda_n(T) \le \sum_{\tilde n=Nr_j+1}^\infty\lambda_{\tilde n}(T) \le (2r_j+1)M_j|U|\delta^j \le (2r_j+1)M_j|U|a^j\left(\frac{r_j}n\right)^{\frac jm}. \]
Simplifying the expression by estimating $r_j\le2(j+m-1)^m\le2(j+m)^m$ and $2r_j+1\le4(j+m)^m$, we finally get that
\[ \lambda_n(T) \le 4|U|M_j\left(a\sqrt[m]{\frac2n}\right)^j(j+m)^{j+m}, \]
where we can choose $j\in\{1,\ldots,k\}$ arbitrary as long as the condition \autoref{eqEvBoundOnRj} is fulfilled.
\end{proof}

In particular, \autoref{thDecayRatesGeneral} includes the trivial case where the kernel $F$ is a polynomial of degree $K$, in which case $M_{K+1}=0$ and we obtain $\lambda_n(T)=0$ for all $n\ge2(a(K+m+1))^m$, since then $K+1\in J_{K+1,n}$. Moreover, we find that for a general $F\in C^k(\bar U\times\bar U)$, we may always pick $j=k$ for $n\ge2(a(k+m))^m$ to obtain that the eigenvalues decay at least as
\[ \lambda_n(T)\le Cn^{-\frac km} \]
for some constant $C>0$.

For smooth kernels $F$, the optimal choice of $j$ depends on the behaviour of the supremum $M_j$ of the directional derivative as a function of $j$.

\begin{corollary}\label{thExponentialDecay}
Let $U\subset\R^m$ be a bounded, open set and $T:L^2(U)\to L^2(U)$ be the positive semi-definite integral operator with the smooth, Hermitian kernel $F\in C^\infty(\bar U\times\bar U)$.

If we have for some constants $B,b,\mu>0$ the inequality
\begin{equation}\label{eqKernelUpperBound}
\frac1{j!}\sup_{x,y\in U}\sup_{v\in S^{m-1}}\left|\Dj F(x,y+sv)\right| \le Bb^jj^{\mu j},
\end{equation}
then there exist constants $C,c>0$ so that the eigenvalues decay at least as
\begin{equation}\label{eqExponentialDecay}
\lambda_n(T)\le Cn\sqrt[m]n\exp\left(-cn^{\frac1{m(1+\mu)}}\right),\quad n\in\N.
\end{equation}
\end{corollary}

\begin{proof}
Using \autoref{thDecayRatesGeneral} with the upper bound \autoref{eqKernelUpperBound} for the constants $M_j$, we find that there exist constants $A,a>0$ so that
\[ \lambda_n(T) \le AB\min_{j+m\le\frac1a\sqrt[m]{\frac n2}}\left[\left(ab\sqrt[m]{\frac2n}\right)^jj^{\mu j}(j+m)^{j+m}\right]\quad\text{for all}\quad n\in\N. \]
To simplify this, we estimate $j^{\mu j}\le(j+m)^{\mu(j+m)}$ and obtain with $\tilde j=j+m$
\begin{equation}\label{eqExponentialDecayMinimum}
\lambda_n(T) \le \tilde An\min_{\tilde j\le\frac1a\sqrt[m]{\frac n2}}\left[\left(ab\sqrt[m]{\frac2n}\right)^{\tilde j}{\tilde j}^{(1+\mu)\tilde j}\right]\quad\text{for all}\quad n\in\N
\end{equation}
for some constant $\tilde A>0$.

To evaluate the minimum in \autoref{eqExponentialDecayMinimum}, we consider for
\begin{equation}\label{eqAlphan}
\alpha_n=ab\sqrt[m]{\frac2n}
\end{equation}
the function
\[ f_n:(0,\infty)\to(0,\infty),\quad f_n(\zeta) = \left(\alpha_n\zeta^{1+\mu}\right)^\zeta. \]
Then, by solving the optimality condition
\[ 0=f_n'(\zeta)=((1+\mu)\log\zeta+1+\mu+\log\alpha_n)f_n(\zeta), \]
we find that $f_n$ attains its minimum at
\begin{equation}\label{eqMinimumF}
\zeta_n=\e^{-1}\alpha_n^{-\frac1{1+\mu}},
\end{equation}
see \autoref{fgFunctionF}.

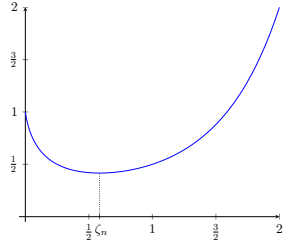
\begin{figure}[ht]
\begin{center}
\begin{tikzpicture}[scale=0.5]
\begin{axis}[axis lines=center,ymin=-0.05,ymax=2,xmin=-0.05,xtick={0,0.5,0.584,1,1.5,2},xticklabels={0,$\frac12$,$\;\zeta_n$,$1$,$\frac32$,$2$},ytick={0,0.5,1,1.5,2},yticklabels={0,$\frac12$,1,$\frac32$,2}]
\addplot[domain=0:2,samples=200,color=blue] {0.5^x*x^(1.5*x)};
\draw[densely dotted](axis cs:0.584,-0.01)--(axis cs:0.584,0.416);
\end{axis}
\end{tikzpicture}
\end{center}
\caption{Graph of the function $f_n(\zeta)=(\alpha_n\zeta^{1+\mu})^\zeta$ for $\alpha_n=\frac12$ and $\mu=\frac12$.}
\label{fgFunctionF}
\end{figure}

Since we only need the asymptotic behaviour for $n\to\infty$, let us pick a value $n_0\in\N$ such that
\begin{alignat*}{2}
\zeta_n&>1,\qquad&&\text{that is}\quad\alpha_n<\e^{-(1+\mu)}<1,\quad\text{and}\\
\zeta_n&<\frac1a\sqrt[m]{\frac n2}=\frac b{\alpha_n},\qquad&&\text{that is}\quad \alpha_n^{\frac\mu{1+\mu}}<b\e,
\end{alignat*}
for all $n\ge n_0$. This can be always achieved since $\alpha_n\to0$ as $n\to\infty$.

Now, the minimum in \autoref{eqExponentialDecayMinimum} is restricted to the set of natural numbers $\tilde j\le\frac1a\sqrt[m]{\frac n2}$ so that we cannot simply insert for $\tilde j$ the minimum point $\zeta_n$ of the function $f_n$. Instead, we estimate the minimum from above by the value of $f_n$ at the largest integer $\lfloor\zeta_n\rfloor$ below $\zeta_n$:
\[ \lambda_n(T) \le \tilde An\min_{\tilde j\le\frac1a\sqrt[m]{\frac n2}}f_n(\tilde j) \le \tilde Anf_n(\lfloor\zeta_n\rfloor). \]

Since $\alpha_n<1$ and $\zeta_n>1$ for all $n\ge n_0$, we get from the explicit formula \autoref{eqMinimumF} for $\zeta_n$ that
\[ f(\lfloor\zeta_n\rfloor)\le\alpha_n^{\zeta_n-1}\zeta_n^{(1+\mu)\zeta_n}=\frac1{\alpha_n}(\alpha_n\zeta_n^{1+\mu})^{\zeta_n}=\frac1{\alpha_n\e^{(1+\mu)\zeta_n}}. \]
Thus, using the expressions \autoref{eqAlphan} and \autoref{eqMinimumF} for $\alpha_n$ and $\zeta_n$, we find constants $C,c>0$ such that
\[ \lambda_n(T) \le Cn\sqrt[m]n\exp\left(-cn^{\frac1{m(1+\mu)}}\right)\quad\text{for all}\quad n\in\N. \]
\end{proof}

\section{Estimating the Kernel of the Integrated Photoacoustic Operator}\label{seDerivativeGreensFunction}

To be able to use \autoref{thExponentialDecay} to estimate the eigenvalues of the operator $\PO^*\PO$, we need to find an upper bound for the derivatives of the integral kernel $F_\kappa$ of the operator $\PO^*\PO$, which is given by \autoref{eqPATOp2Kernel}. Since
\[ F_\kappa(x,y) = \int_{-\infty}^\infty\int_{\partial\Omega}\frac1{\omega^2}\overline{G_\kappa(\omega,\xi-x)}G_\kappa(\omega,\xi-y)\d S(\xi)\d\omega, \]
where $G_\kappa$ denotes the fundamental solution of the Helmholtz equation, given by \autoref{eqSolAttWaveKernel}, we start with the directional derivatives of the function $G_\kappa(\omega,\cdot)$. Since $G_\kappa(\omega,\cdot)$ is radially symmetric, this means we can write it for arbitrary $\omega\in\R$ and $x\in\R^3\setminus\{0\}$ in the form
\begin{equation}\label{eqGreensFunctionRadialPart}
G_\kappa(\omega,x) = g_{\kappa,\omega}(\tfrac12|x|^2)\quad\text{with}\quad g_{\kappa,\omega}(\rho) = -\frac{\i \omega}{4\pi\sqrt{2\pi}}\,\frac{\e^{\i\kappa(\omega)\sqrt{2\rho}}}{\sqrt{2\rho}}\quad\text{for}\quad\rho>0,
\end{equation}
this problem reduces to the calculation of one dimensional derivatives of the function $g_{\kappa,\omega}$.

\begin{lemma}\label{thDerivativeRadialSymmetric}
Let $\phi\in C^\infty(\R)$ be defined by
\[ \phi(s)=\frac12|x+sv|^2 \]
for some arbitrary $x\in\R^m$ and $v\in S^{m-1}$. Then, we have for every function $\gamma\in C^\infty(\R)$ that
\begin{equation}\label{eqDerivativeRadialSymmetric}
(\gamma\circ\phi)^{(j)}(0) = \sum_{k=0}^{\lfloor\frac j2\rfloor}\frac{j!}{2^kk!(j-2k)!}\left<v,x\right>^{j-2k}\gamma^{(j-k)}(\tfrac12|x|^2).
\end{equation}
\end{lemma}
\begin{proof}
Since $\phi'(0)=\left<v,x\right>$, $\phi''(0)=1$, and all higher derivatives of $\phi$ are zero, the formula of Faà di Bruno, see for example~\cite[Chapter 3.4, Theorem A]{Com74}, simplifies to
\[ (\gamma\circ\phi)^{(j)}(0) = \sum_{\alpha\in A_{2,j}}\frac{j!}{\alpha_1!\alpha_2!}\left(\frac{\left<v,x\right>}{1!}\right)^{\alpha_1}\left(\frac1{2!}\right)^{\alpha_2}\gamma^{(\alpha_1+\alpha_2)}(\tfrac12|x|^2), \]
where $A_{2,j}=\{\alpha\in\N_0^2\mid\alpha_1+2\alpha_2=j\}$. Setting $k=\alpha_2$ and thus $\alpha_1=j-2k$, we obtain the formula in the form \autoref{eqDerivativeRadialSymmetric}.
\end{proof}

Thus, the directional derivatives of $G_\kappa(\omega,\cdot)$ can be calculated from the derivatives of $g_{\kappa,\omega}$, which we may estimate directly.
\begin{lemma}\label{thDerivativeOfExponential}
Let $\gamma_a\in C^\infty((0,\infty))$ denote the function
\begin{equation}\label{eqGammaA}
\gamma_a(\rho) = \frac{\e^{a\sqrt{2\rho}}}{\sqrt{2\rho}},\quad\rho>0,\;a\in\C.
\end{equation}
Then, we have for every $j\in\N_0$ and all $\rho>0$ the inequality
\begin{equation}\label{eqDerivativeOfExponential}
|\gamma_a^{(j)}(\rho)| \le 2^j(j+1)!\left(\e^{j+1}+\frac1{j!}\left(\frac\rho2\right)^{\frac j2}|a|^j\right)\frac{|\gamma_a(\rho)|}{\rho^j}.
\end{equation}
\end{lemma}
\begin{proof}
Let us first assume that $a\ne0$ and write $\gamma_a(\rho) = \frac1a\frac{\d}{\d\rho}\e^{a\sqrt{2\rho}}$. Thus, we have for every $j\in\N_0$ that
\[ \gamma_a^{(j)}(\rho) = \frac1a\frac{\d^{j+1}}{\d\rho^{j+1}}\e^{a\sqrt{2\rho}}. \]
Applying to this the formula of Faà di Bruno, see for example~\cite[Chapter 3.4, Theorem A]{Com74}, we find with $A_{j+1}=\{\alpha\in\N_0^{j+1}\mid\sum_{k=1}^{j+1}k\alpha_k=j+1\}$ that
\begin{align*}
\gamma_a^{(j)}(\rho) &= \sum_{\alpha\in A_{j+1}}\frac{(j+1)!}{\alpha!}(2\rho)^{-\frac{\alpha_1}2}\prod_{k=2}^{j+1}\left((-1)^{k+1}\frac{(2k-3)!!}{k!}(2\rho)^{\frac12-k}\right)^{\alpha_k}a^{|\alpha|-1}\e^{a\sqrt{2\rho}} \\
&= \sum_{\alpha\in A_{j+1}}\frac{(j+1)!}{\alpha!}\prod_{k=2}^{j+1}\left((-1)^{k+1}\frac{(2k-3)!!}{k!}\right)^{\alpha_k}(2\rho)^{\frac12|\alpha|-j-1}a^{|\alpha|-1}\e^{a\sqrt{2\rho}}.
\end{align*}
This formula also extends continuously to the case $a=0$.

Estimating it from above by using that $\frac{(2k-3)!!}{k!}\le\frac{2^{k-1}(k-1)!}{k!}\le 2^{k-1}$, we obtain that
\[ |\gamma_a^{(j)}(\rho)| \le \sum_{\alpha\in A_{j+1}}\frac{(j+1)!}{\alpha!}\left(\frac\rho2\right)^{\frac12|\alpha|}|a|^{|\alpha|-1}\frac{\e^{\Re a\sqrt{2\rho}}}{\rho^{j+1}}. \]

Now, using the combinatorial identity
\[ \sum_{\alpha\in A_{j+1}\cap P_{j+1,\ell+1}}\frac{(j+1)!}{\alpha!}=\binom{j}{\ell}\frac{(j+1)!}{(\ell+1)!} \]
for $P_{j+1,\ell+1}=\{\alpha\in\N_0^{j+1}\mid|\alpha|=\ell+1\}$, see for example~\cite[Chapter 3.3, Theorem B]{Com74}, we find that
\[ |\gamma_a^{(j)}(\rho)|\le\sum_{\ell=0}^j\binom{j}{\ell}\frac{(j+1)!}{(\ell+1)!}\left(\frac\rho2\right)^{\frac{\ell+1}2}|a|^\ell\frac{\e^{\Re a\sqrt{2\rho}}}{\rho^{j+1}}. \]
We may further estimate this by using $\binom{j}{\ell}\le\sum_{k=0}^j\binom{j}{k}=2^j$ and
\begin{equation}\label{eqPartialExponentialSum}
\sum_{\ell=0}^j\frac{s^\ell}{(\ell+1)!}\le\e^{j+1}+\frac{s^j}{j!}\quad\text{for every}\quad s>0
\end{equation}
(since $\sum_{\ell=0}^j\frac{s^\ell}{(\ell+1)!}\le\sum_{\ell=0}^j\frac{s^j}{(j+1)!}=\frac{s^j}{j!}$ if $s\ge j+1$ and $\sum_{\ell=0}^j\frac{s^\ell}{(\ell+1)!}\le\sum_{\ell=0}^j\frac{(j+1)^\ell}{\ell!}\le\e^{j+1}$ otherwise) to obtain~\autoref{eqDerivativeOfExponential}.
\end{proof}

Putting together \autoref{thDerivativeRadialSymmetric} and \autoref{thDerivativeOfExponential}, we find an estimate for the directional derivatives of the function $G_\kappa$.

\begin{proposition}\label{thGreensFunctionDirectionalDerivative}
Let $G_\kappa$ be given by \autoref{eqSolAttWaveKernel} for some arbitrary function $\kappa:\R\to\C$. Then, there exists a constant $C>0$ so that we have for every $j\in\N_0$, $x\in\R^3\setminus\{0\}$, and $v\in S^2$ the inequality
\begin{equation}\label{eqGreensFunctionDirectionalDerivative}
\frac1{j!}\left|\Dj G_\kappa(\omega,x+sv)\right|\le |G_\kappa(\omega,x)|C^j\left(\frac1{|x|^j}+\frac1{j!}|\kappa(\omega)|^j\right).
\end{equation}
\end{proposition}
\begin{proof}
Writing $G_\kappa$ in the form \autoref{eqGreensFunctionRadialPart}, \autoref{thDerivativeRadialSymmetric} implies (with $\gamma=g_{\kappa,\omega}$) that
\[ \left|\Dj G_\kappa(\omega,x+sv)\right| \le \sum_{k=0}^{\lfloor\frac j2\rfloor}\frac{j!}{2^kk!(j-2k)!}|x|^{j-2k}\,|g_{\kappa,\omega}^{(j-k)}(\tfrac12|x|^2)|. \]
Inserting the estimate for $g_{\kappa,\omega}^{(j-k)}$ obtained from \autoref{thDerivativeOfExponential} (using that $g_{\kappa,\omega}=-\frac{\i \omega}{4\pi\sqrt{2\pi}}\gamma_{\i\kappa(\omega)}$ with $\gamma_{\i\kappa(\omega)}$ being defined by \autoref{eqGammaA} and evaluating at $\rho=\frac12|x|^2$), we find that
\[ \left|\Dj G_\kappa(\omega,x+sv)\right| \le j!\frac{|G_\kappa(\omega,x)|}{|x|^j}\sum_{k=0}^{\lfloor\frac j2\rfloor}\frac{(j-k+1)!}{k!(j-2k)!}2^{2j-3k}\left(\e^{j-k+1}+\frac1{(j-k)!}\left(\frac{|x|\,|\kappa(\omega)|}2\right)^{j-k}\right). \]
Using further that $\frac{(j-k+1)!}{k!(j-2k)!}=(j-k+1)\binom{j-k}k\le(j-k+1)2^{j-k}$, we find that there exists a constant $\tilde C>0$ so that
\[ \frac1{j!}\left|\Dj G_\kappa(\omega,x+sv)\right| \le \frac{|G_\kappa(\omega,x)|}{|x|^j}\tilde C^j\left(1+\sum_{k=\lceil\frac j2\rceil}^j\frac1{k!}\left(\frac{|x|\,|\kappa(\omega)|}2\right)^k\right). \]

Estimating the sum herein by using relation \autoref{eqPartialExponentialSum}, we obtain the inequality \autoref{eqGreensFunctionDirectionalDerivative}.
\end{proof}

\autoref{thGreensFunctionDirectionalDerivative} allows us to estimate the derivatives of the function $G_\kappa$, however, to apply \autoref{thExponentialDecay} to $\PO^*\PO$ for the integrated photoacoustic operator $\PO$, we need to estimate the derivatives of the kernel~$F_\kappa$ of $\PO^*\PO$, given by \autoref{eqPATOp2Kernel}. For the integral over the frequency which appears in this estimate, we use the following result in the proof of \autoref{thForwardOperatorKernelUpperBound}.

\begin{lemma}\label{thFrequencyIntegral}
Let $\varepsilon>0$ and $\kappa:\R\to\bar\H$ be a measurable function fulfilling the inequality \autoref{eqAttCoeffPolBdd} for $\ell=0$ with some constants $\kappa_1>0$ and $N\in\N$ and the inequality \autoref{eqStrongAttenuation} with some constants $\omega_0>0$, $\kappa_0>0$, and $\beta>0$.

Then, there exist constants $B,b>0$ so that we have for every $j\in\N_0$ the estimate
\[ \frac1{j!}\int_{-\infty}^\infty|\kappa(\omega)|^j\e^{-2\varepsilon\Im\kappa(\omega)}\d\omega \le Bb^jj^{(\frac{N}{\beta}-1)j}.  \]
\end{lemma}
\begin{proof}
By our assumptions on $\kappa$, we have that there is a constant $C\ge0$ such that $\Im\kappa(\omega)\ge\kappa_0|\omega|^\beta-C$ for all $\omega\in\R$. Thus,
\[ \frac1{j!}\int_{-\infty}^\infty|\kappa(\omega)|^j\e^{-2\varepsilon\Im\kappa(\omega)}\d\omega\le\frac{2\e^{2C\varepsilon}\kappa_1^j}{j!}\int_0^\infty(1+\omega)^{Nj}\e^{-2\varepsilon\kappa_0\omega^\beta}\d\omega. \]

By Jensen's inequality, applied to the convex function $f(r)=r^{Nj}$, we can estimate
\[ (1+\omega)^{Nj} = 2^{Nj}\left(\frac{1}{2}+\frac{1}{2}\omega\right)^{Nj}\le 2^{Nj-1}(1+\omega^{Nj}). \]
Then, we find with the substitution $\nu=\omega^\beta$ that
\begin{align*}
\frac1{j!}\int_{-\infty}^\infty|\kappa(\omega)|^j\e^{-2\varepsilon\Im\kappa(\omega)}\d\omega &\le \frac{(2^N\kappa_1)^j\e^{2C\varepsilon}}{\beta j!}\int_0^\infty\nu^{\frac1\beta-1}(1+\nu^{\frac{Nj}\beta})\e^{-2\varepsilon\kappa_0\nu}\d \nu \\
&= \frac{(2^N\kappa_1)^j\e^{2C\varepsilon}}{\beta\Gamma(j+1)}\left((2\varepsilon\kappa_0)^{-\frac1\beta}\Gamma(\tfrac1\beta)+(2\varepsilon\kappa_0)^{-\frac{Nj+1}\beta}\Gamma(\tfrac{Nj+1}\beta)\right),
\end{align*}
where
\[ \Gamma(\rho) = \int_0^\infty \nu^{\rho-1}\e^{-\nu}\d\nu,\quad \rho>0, \]
denotes the gamma function.

Recalling Stirling's formula, see for example \cite[Section 6.1.42]{AbrSte72}, we know that the gamma function can be bounded from below and above by
\[ \sqrt{\frac{2\pi}\rho}\left(\frac\rho\e\right)^\rho \le \Gamma(\rho) \le \sqrt{\frac{2\pi}\rho}\left(\frac\rho\e\right)^\rho\e^{\frac1{12\rho}}\quad\text{for every}\quad\rho>0. \]
Thus, we find constants $B,b>0$ so that
\[ \frac1{j!}\int_{-\infty}^\infty|\kappa(\omega)|^j\e^{-2\varepsilon\Im\kappa(\omega)}\d\omega \le Bb^jj^{(\frac{N}{\beta}-1)j}\]
for all $j\in\N_0$.
\end{proof}

\section*{Acknowledgement}
The work is support by the ``Doctoral Program Dissipation and Dispersion in Nonlinear PDEs'' (W1245). OS is also supported by the FWF-project ``Interdisciplinary Coupled Physics Imaging'' (FWF P26687).

The authors thank Adrian Nachman (from the University of Toronto) for stimulating discussions on pseudodifferential operators.

\section*{References}
\renewcommand{\i}{\ii}
\printbibliography[heading=none]

\end{document}